 \pdfoutput=1

\documentclass[11pt]{amsart}
\usepackage{amsmath,amssymb}
\usepackage{graphicx}
\usepackage{amsfonts,amscd,latexsym}

\newcommand{\labell}[1] {\label{#1}}

\newlength{\facewd} \newlength{\faceht}%

\newcommand{\ND}{{\mathcal {ND}}}
\newcommand{\Starr}{{\rm Star}}
\newcommand{\starr}{{\rm star}}
\newcommand{\istar}{{{ {\rm Star^*} }}}
\newcommand{\pr}{{\rm pr}}
\newcommand{\Isom}{{\rm Isom\,}}

\renewcommand{\Hat}{\widehat}
\newcommand{\intt}{{\rm int\,}}

\newcommand{\less}{{\smallsetminus}}
\newcommand{\sspan}{{\rm span}}

\newcommand{\p}{{\partial}}
\newcommand{\al}{{\alpha}}

\newcommand{\be}{{\beta}}

\newcommand{\om}{{\omega}}
\newcommand{\eps}{{\varepsilon}}
\newcommand{\vareps}{{\epsilon}}

\newcommand{\De}{{\Delta}}

\newcommand{\ka}{{\kappa}}
\newcommand{\la}{{\lambda}}
\newcommand{\La}{{\Lambda}}

\newcommand{\Cc}{{\mathcal C}}

\newcommand{\Ss}{{\mathcal S}}

\newcommand{\ov}{\overline}

\newcommand{\aff}{{\it aff}}

\renewcommand{\Tilde}{\widetilde}

\newcommand{\ft}{{\mathfrak t}}

\newcommand{\PP}{{\mathbb P}}

\newcommand{\R}{{\mathbb R}}
\newcommand{\C}{{\mathbb C}}

\newcommand{\Z}{{\mathbb Z}}

\newcommand{\Ham}{{\rm Ham}}

\newcommand{\Vv}{{\mathcal V}}

\newcommand{\SSS}{{\smallskip}}

\newtheorem{theorem}{Theorem}[section]
\newtheorem{thm}[theorem]{Theorem}

\newtheorem{cor}[theorem]{Corollary}
\newtheorem{lemma}[theorem]{Lemma}

\newtheorem{prop}[theorem]{Proposition}

\newtheorem{defn}[theorem]{Definition}
\newtheorem{example}[theorem]{Example}

\newtheorem{rmk}[theorem]{Remark}

\numberwithin{figure}{section}
\numberwithin{equation}{section}
\numberwithin{table}{section}

\newcommand{\MS}{{\medskip}}

\newcommand{\NI}{{\noindent}}

\begin{document}

\title{Displacing Lagrangian toric fibers via probes}
\author{Dusa McDuff}\thanks{Partially supported by NSF grants DMS 0604769
and DMS 0905191}
\address{Department of Mathematics,
Barnard College, Columbia University}
\email{dusa@math.columbia.edu}

\date{October 26, 2009, rev. February 25, 2010.}

%
 
 \begin{abstract}  This note studies the geometric structure of monotone  moment polytopes (the duals of smooth Fano polytopes) using probes.  The latter  are line segments that enter the polytope at an interior point of a facet and whose 
direction is  integrally transverse to this facet.  A point inside the polytope is displaceable by a probe if it lies less than half way along it.  Using a construction due to
Fukaya--Oh--Ohta--Ono, we show that
every rational polytope has a central point that is not displaceable by probes.
In the monotone case, this central point is its unique interior integral point, and
 we show that every other point is displaceable by probes if and only if the  polytope satisfies the star Ewald condition.  (This is a strong version of the Ewald conjecture concerning the integral symmetric points 
in the polytope.)  
Further, in dimensions up to and including three
every monotone polytope is star Ewald.  These results are closely related to
the Fukaya--Oh--Ohta--Ono calculations of the Floer homology
of the Lagrangian fibers of a toric symplectic manifold, and have applications to questions introduced by Entov--Polterovich about the displaceability of these fibers.
 \end{abstract}
 
 \maketitle
\section{Introduction}\labell{s:intro}

Symplectic toric manifolds form a nice family of examples in which to test various ideas; they can be described by simple diagrams and many of their invariants can be explicitly calculated. In this paper we  discuss some geometric problems that 
 arise when studying their natural family of Lagrangian submanifolds.  
 Before describing our results, we shall illustrate the main   definitions by some examples.

A symplectic toric  manifold is a $2n$-dimensional symplectic manifold with
an action of the $n$-torus $T^n$ that is Hamiltonian, i.e. the action of the 
$i$th component circle is induced by a function $H_i:M\to \R$.
These functions $H_i$ fit together to give the moment map 
$$
\Phi:M\to \R^n,\quad \Phi(x) = \bigl(H_1(x),\dots,H_n(x)\bigr),
$$
whose image is called the {\it moment polytope}.
The first examples are:\MS

\NI
$\bullet$  $S^2$ with its standard 
area form (normalized to have area $2$) and with the $S^1$ action that rotates about the north-south axis; the corresponding Hamiltonian is the height function, and its moment polytope
 is the $1$-simplex $[-1,1]$; see Fig. \ref{fig:11}.\SSS

\begin{figure}[htbp] 
   \centering
 \includegraphics[width=3in]{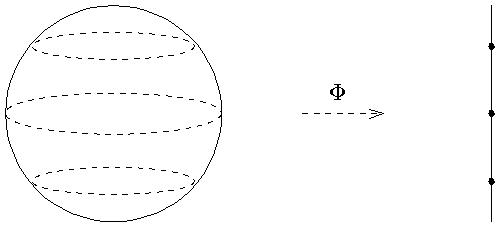} 
   \caption{The $S^1$ action on $S^2$ with some of its orbits.  Clearly  each of them except for the equator can be displaced by an area preserving isotopy.}
   \label{fig:11}
\end{figure}

\NI $\bullet$  $\R^{2n}=\C^n$ with the usual  symplectic form,
normalized as 
$$
\om: = \sum \frac i{2\pi} dz_j\wedge d{\ov z}_j = \frac 1{\pi}\sum_{k=1}^{n} dx_{2k-1}\wedge dx_{2k},
$$
and $T^n$ action $(z_1,\dots,z_n)\mapsto (e^{2\pi i t_1} z_1,\dots,e^{2\pi i t_n}z_n)$ for $t_j\in S^1\equiv \R/\Z$.  Then the moment map is 
\begin{equation}\label{eq:Rn}
\Phi_0: (z_1,\dots,z_n)\mapsto \Bigl(|z_1|^2,\dots,|z_n|^2\Bigr)\in \R^n,
\end{equation}
with image the first quadrant $\R^n_+$ in $\R^n$. Note that the inverse image under $\Phi$ of the line segment $(x, c_2,\dots,c_n),  0\le x\le a$, where $c_j>0$ for $j>1$,
is the product $D^2(a)\times T^{n-1}$,  where $D^2(a)$ is a disc of $\om$-area $a$ in the first factor $\C$ and the torus is a product of the
circles  $|z_j|=\sqrt c_j$ in the other copies of $\C$.

\SSS

\NI $\bullet$  the projective plane $\C P^2$ with the Fubini--Study form
and $T^2$ action $[z_0;z_1;z_2]\mapsto [z_0;\la_1z_2;\la_2z_2]$.
The moment map is 
$$
[z_0;z_1;z_2]\mapsto  \bigl(\frac{|z_1|^2}{\sum_i|z_i|^2},
\frac{|z_2|^2}{\sum_i|z_i|^2}\bigr),
$$
with image the standard triangle $\{x_1,x_2\ge 0, x_1+x_2\le 1\}$ in $\R^2$.\MS
\SSS

More moment polytopes are illustrated in Figures \ref{fig:probe0} and \ref{fig:7} below.  Good references are Audin \cite{Au} and Karshon--Kessler--Pinsonnault \cite{KKP}.  Notice that the moment map simply 
quotients out by the $T^n$ action.

One important fact here is that the moment polytope
$\Phi(M)$ is a convex polytope $\De$, satisfying certain integrality conditions at each vertex.  (We give a precise statement in Theorem \ref{thm:delz} below.) Another is that the symplectic form on $M$ is determined by the polytope $\De$; indeed every point in $(M,\om)$ has a Darboux chart that is equivariantly symplectomorphic to a set of the form $\Phi_0^{-1}(V)$, where $V$
is a neighborhood of some point in the first quadrant $\R^n_+$.   Thus locally the action looks like that of $T^n$ in $\C ^n$.  In particular the (regular) orbits of $T^n$ are Lagrangian.\footnote
{One way to prove this is to note that  the functions $H_i$ are in involution; 
the Poisson brackets $\{H_i,H_j\}$ vanish because $T^n$ is abelian.} 
Hence
 the inverse image of each interior point $u\in \intt\De$ is a smooth Lagrangian manifold $L_u$, that is $\om|_{L_u} = 0$.
 
This note addresses the question of which of these toric fibers $L_u$ are {\it displaceable by a Hamiltonian isotopy,}  i.e. are such that there is a family of functions $H_t:M\to \R, t\in [0,1],$ whose associated flow $\phi_t, t\in [0,1],$ 
has the property that
$\phi_1(L_u)\cap L_u=\emptyset$.

This question was first considered (from different points of view)
in \cite{BEP} and  \cite{Cho}.
Biran, Entov and Polterovich  showed in \cite{BEP}
 that if the quantum homology $QH_*(M)$ 
(taken with appropriate coefficients) has a field summand  
then at least one of the fibers $L_u$ is 
nondisplaceable. In later work (cf. \cite[Theorem~2.1]{EP2}), Entov--Polterovich managed to dispense with the condition on quantum homology.
Even more recently,
they showed in \cite[Theorem~1.9]{EPrig} by a dynamical argument that if in addition $(M,\om)$ is {\it monotone}, i.e. $[\om]$ is a positive multiple of the first Chern class $c_1(M)$, then   
 the fiber over the so-called {\it special point} $u_0$ of $\De$ 
is nondisplaceable.\footnote
{If $\om$ is  normalized so that $[\om]=c_1(M)$, 
then the smooth moment polytope $\De$ is dual to an integral  Fano polytope $P$ and $u_0$ is its unique interior integral point which is usually placed at $\{0\}$.  We call such polytopes $\De$ {\it monotone};
 cf. Definition \ref{def:mono}.}   
However, for general polytopes their argument
gives no information about which fibers might be nondisplaceable.

Cho \cite{Cho}, and later 
 Fukaya, Oh, Ohta and Ono \cite{FOOO1,FOOO2}, took a 
more constructive approach to this problem.
The upshot of this work is that 
for any toric manifold,
one can define
 Floer homology groups 
$HF^*(L_u,\chi)$  
(depending on various deformation parameters $\chi$) that
 vanish whenever $L_u$ is displaceable.  
Moreover, in \cite[\S9]{FOOO1}, the authors  construct
 a point $v_0$
for which this Floer homology does not vanish for suitable $\chi$, at least in the  case
when $[\om]$ is a rational cohomology class. 
 They show in \cite[Thm.~1.5]{FOOO1} that even in the nonrational case the corresponding fiber $L_{v_0}$ cannot be displaced.  They also show in the monotone case  that $v_0$ coincides with the special point $u_0$ and that
 $HF^*(L_u,\chi)=0$ for all other $u$;  cf. \cite[Thm.~7.11]{FOOO1}.

One of the main motivating questions for the current study
 was raised by
Entov and Polterovich, who ask in \cite{EPrig} whether the special fiber 
$L_{u_0}$ is a {\it stem}, that is, whether all {\it other } fibers are displaceable.  The results stated above show that from the Floer theory point of view this holds.

In this paper we develop  a geometric way
to displace toric fibers using probes.   Our method is based on
the geometry of the moment  polytope $\De$ and makes sense 
for general rational polytopes.
Using it, we show:

\begin{thm}\labell{thm:3}   If $(M,\om)$ is a monotone toric symplectic manifold of (real) dimension $\le 6$ then the special fiber $L_{u_0}$ is a stem.
\end{thm}

This is an immediate consequence of Proposition \ref{prop:3}.

 In our approach, this question  of which toric fibers can be displaced
 is closely related to the well known Ewald conjecture in \cite{Ew} about the structure of monotone
polytopes $\De$, namely that the set 
$$
\Ss(\De) = \{v\in \De\cap \Z^n\,:\,-v\in \De\}
$$
 of symmetric integral points in $\De$ 
contains an integral basis for $\R^n$.\footnote
{
As is customary in this subject, Ewald works with the dual polytope $P$
that is constructed from the fan.  Hence his formulation looks very different from ours, but is equivalent.}
  By work of 
{\O}bro \cite[Ch.~4.1]{Ob}, this 
is now known to hold in dimensions $\le 8$.  
However, in general it is not 
even known if $\Ss(\De)$ must be nonempty.

In  Definition \ref{def:starE} we formulate a stronger, but still purely combinatorial,
 version of the Ewald property (called the {\it star Ewald} property) 
 and prove the following result.
 
 \begin{thm}\labell{thm:stEmain} 
  A monotone polytope $\De$ 
satisfies the  star Ewald condition
if and only if
 every point in $\intt\De\less\{0\}$ can be displaced by a probe. 
 \end{thm}

 \begin{cor}\labell{cor:stE}
   If $\De$ is a monotone polytope in $\R^n$  
 for which all points except for
 ${u_0}$ are displaceable by probes, 
 then $\Ss(\De)$ spans $\R^n$. 
 \end{cor}

 We  show in Proposition \ref{prop:3} that every 
 $3$-dimensional monotone polytope satisfies the star Ewald condition.  
The proof of this result
is fairly geometric, and does not appear to generalize easily to higher dimensions. 
Therefore, before attempting such a generalization,
 it would seem sensible to carry out a  computer check of
  the star Ewald condition
 using {\O}bro's methods.\footnote
 {
 A. Paffenholz  has recently made such a computer search, finding 
 that the first dimension in which  counterexamples occur is  $6$. }

 We then analyze the star Ewald 
 condition for monotone polytopes  that are bundles.  
 (Definitions are given in \S\ref{s:bun}.)  By Lemma \ref{le:bun}  
 the fiber and base of such a bundle must be monotone.
 Although it seems likely that the total space is star Ewald whenever the fiber and base are, we could only prove the following special case.

 \begin{prop} \labell{prop:bun} Suppose that the monotone polytope $\De$ is a  bundle over
 the $k$-simplex $\De^k$, whose fiber $\Tilde \De$ satisfies the 
 star Ewald condition.  Then $\De$ satisfies the 
  star Ewald condition.
  \end{prop}
  
   Using this, we   show that $L_{u_0}$ is a stem in various other cases, in particular for the $8$-dimensional monotone manifold found by Ostrover--Tyomkin \cite{OT}
that does not have semisimple quantum homology.  

During the course of the proof we show that if the polytope 
 $\Hat \De$  is  star Ewald, then the total space of every bundle over  $\Hat \De$ with
 star Ewald fiber is itself
 star Ewald if and only if this is true for
bundles with fiber the one-simplex $\De_1$; see Proposition \ref{prop:bun1}.

Finally, we discuss the notion of stable displaceability 
in \S\ref{ss:stdis}. This notion was introduced by
Entov--Polterovich in \cite{EPrig} as an attempt to generalize the notion of displaceability.  However, we show in
Proposition \ref{prop:stdis} that in many cases stably displaceable fibers are actually displaceable by probes.

Our arguments rely on the Fukaya--Oh--Ohta--Ono
notion of the {\it central  point} $v_0$.  We explain this in 
 \S \ref{ss:v0}, and then in 
Lemma \ref{le:v0} give a direct combinatorial proof of the following fact.

\begin{prop}\labell{prop:uF} For
 every rational polytope the point  ${v_0}$ is not displaceable by probes.
\end{prop}

In some cases, it is easy to check  that probes 
displace all points $u\in \intt \De $ for which $HF^*(L_u,\chi)=0$.
For example, the results of 
Fukaya--Oh--Ohta--Ono concerning  one point blow ups of $\C P^2$ and  $\C P^1\times \C P^1$
become very clear from this perspective: see Example \ref{ex:2blow} and 
Figure \ref{fig:probe0}. 
However,  by Lemma \ref{le:inacc}
this is no longer true  for general Hirzebruch surfaces.
Further,
the Floer-theoretic nondisplaceable  set 
$$
\ND_{HF}: = \{u\in \intt\De:HF^*(L_u,\chi)\ne0\}
$$
has dimension at most $n-1$, while we prove the following result 
in \S\ref{ss:2}.

\begin{prop}\labell{prop:notdisp} There are
$2$-dimensional
 moment polytopes with a nonempty open set consisting of fibers that are not displaceable by probes.
\end{prop}

It is not at all clear whether these fibers  really are nondisplaceable, or
 whether one just needs to find more elaborate 
 ways of displacing them.
\MS

\NI {\bf Varying the symplectic form}

In general, the set of nondisplaceable fibers varies as one varies the toric symplectic form.  In terms of the moment polytope this amounts to
changing the support constants $\ka_i$ of the facets $\{x:\langle x,\eta_i\rangle \le \ka_i\}$  without changing the normal vectors $\eta_i$.
For each $\De$ we denote by $\De(\ka)$ the polytope with support constants $\ka$ and normals equal to those of $\De$, and 
 by $\Cc_\De$  the set of $\ka = (\ka_i)$ for which 
 $\De(\ka)$ is analogous to $\De$, i.e. is such that a set of facets has nonempty intersection in $\De(\ka)$ if and only if it does in $\De$.
Let us say that $\De(\ka)$ is {\it accessible} if all its points except for $v_0$ are displaceable by probes.  Then we may ask:

\begin{itemize} \item For which $\De$ is there some 
$\ka\in \Cc_\De$ such that 
$\De(\ka)$ is accessible?
\item For which $\De$ is $\De(\ka)$ accessible for all $\ka\in \Cc_\De$?
\end{itemize}

If $\De$ is a product of simplices, it is obvious that
$\De(\ka)$ is always accessible.  However in dimension $2$ some 
trapezoids (the even ones) are also accessible for all $\ka$; see 
Corollary~\ref{cor:inacc}. It is not clear what happens in higher dimensions.
\MS

\NI {\bf Acknowledgements.}  This paper grew out of an attempt with
Leonid Polterovich and Misha Entov to understand the displaceability of fibers of $2$-dimensional monotone polytopes, and  
I wish to thank them for useful discussions.  I also am very grateful to
Benjamin  Nill for his many penetrating comments on earlier drafts of this note and, in particular, for sharpening the original version of Lemma \ref{le:span}.
 Discussions with Fukaya, Ohta and Ono and with 
 Chris Woodward helped to clarify some of 
 the examples in \S \ref{ss:2}.
 Finally,  I would like to thank the referee for
reading the manuscript so carefully and pointing out 
many small inaccuracies.

\section{The method of probes}\labell{s:probe}

\subsection{Basic notions}\labell{ss:pr}

A line is called rational if its direction vector is rational.
The {\it affine distance} $d_{\aff}(x,y)$ between two  points $x,y$ on a rational line $L$
 is the ratio of their Euclidean distance $d_{E}(x,y)$  to the  minimum
  Euclidean distance from $0$  to an integral point on the line through $0$ 
  parallel to $L$.
  Equivalently, if $\phi$ is any integral affine transformation of $\R^n$  that takes $x,y$ to the $x_1$ axis, then 
$d_{\aff}(x,y)=d_{E}(\phi x, \phi y)$.  
  
An affine hyperplane $A$ is called {\it rational} if it has a primitive integral normal vector $\eta$, i.e. if it is given by an equation of the form  $\langle x,\eta\rangle =\ka$ where $\ka\in \R$ and $\eta$ is 
primitive and integral.
The {\it affine distance } $d_{\la}(x,A)$ from a point $x$ to a (rational) affine  hyperplane $A$ in the (rational) direction $\la$ is defined to be 
\begin{equation}\labell{eq:aff}
d_\la(x,A): = d_{\aff}(x,y)
\end{equation}
 where 
$y\in A$ lies on the ray $x+a\la, a\in \R^+$. (If this ray does not meet  $A$, 
we set $d_\la(x,A)=\infty$.)
  If the direction $\la$ is not specified, we take it to be $\eta$.
We shall say that an integral vector $\la$ is 
{\it integrally transverse to} $A$ if it  can be completed to an integral basis  by vectors parallel to $A$.    Equivalently, we need 
$|\langle \la,\eta\rangle|=1$ where $\eta$ is the normal as above.  

The next lemma shows that  the affine distance of $x$ from $A$ is {\it maximal} along these affine transverse directions. If $A = \{\langle x,\eta\rangle =\ka\}$,
we define
$\ell_A(x): = \ka -\langle x,\eta\rangle$.

\begin{lemma} \labell{le:affd} Let $A$ be the hyperplane $\ell_A(x): = \ka -\langle x,\eta\rangle =0$, where $\eta$ is a primitive integral vector.  Then 
for any rational points $u\notin A$ and $y\in A$
$$
d_{\aff}(u,y)\le |\ell_A(u)|,
$$
with equality if and only if the primitive integral  vector in the direction
  $y-u$ is integrally transverse to $F$.
\end{lemma}

\begin{proof}  This is obvious if one chooses coordinates so that 
$A=\{x_1=0\}.$
\end{proof}

A (convex, bounded) polytope $\De\subset \R^n$  is called  {\it rational} if each of its facets $F_i, i=1,\dots,N,$ is rational.
Thus there are  primitive integral vectors $\eta_i$ (the outward normals) and constants $\ka_i\in \R$ so that
\begin{equation}\labell{eq:De}
\De = \bigl\{x\in \R^n\,|\, \langle\eta_i,x\rangle \le \ka_i, i=1,\dots,N\bigr\}.
\end{equation}
We denote by 
\begin{equation}\labell{eq:li}
  \ell_i:\De\to \R,\quad x\mapsto \ka_i-\langle\eta_i,x\rangle
\end{equation}
  the affine distance from $x\in \De$ to the facet $F_i$. Further, 
$\De$ is {\it simple} if exactly $n$ facets meet at each vertex, and is 
  {\it integral} if its set $\Vv(\De)$ of vertices are integral.  (Integral polytopes are also known as {\it lattice} polytopes.)
 A simple, rational polytope is  
  {\it smooth} if for each vertex $v\in\Vv(\De)$ the normals $\eta_i, i\in I_v,$ of the facets meeting at $v$ form 
  a basis for the integral lattice $\Z^n$. This is equivalent to requiring that for each vertex $v$  the $n$ primitive integral vectors  $e_i(v)$ pointing along the edges from $v$ form a lattice basis.  

Delzant proved the following foundational theorem in \cite{Delz}.

\begin{thm} \labell{thm:delz} There is a bijective correspondence between smooth polytopes
in $\R^n$
(up to integral affine equivalence)  and toric symplectic $2n$-manifolds 
(up to equivariant symplectomorphism).
\end{thm}

\begin{defn}   Let $w$ be a point of some facet $F$ of a rational polytope $\De$ and $\la\in \Z^n$
be  integrally transverse to $F$.  
The {\bf probe} $p_{F,\la}(w) = p_\la(w)$ with direction $\la\in \Z^n$ and initial point $w\in F$ is the half open line segment consisting of $w$ together with the 
points in $\intt\De$ that lie on the ray from $w$ in direction $\la$.
\end{defn}


In the next lemma we can use any notion of length along a line, though 
 the affine distance is the most natural.  

\begin{lemma}\labell{le:probe}  Let $\De$ be a smooth moment polytope.
Suppose that a point $u\in \intt\De$ lies on the probe $p_{F,\la}(w)$.  Then if  $w$ lies in the interior of $F$ and $u$ is less than halfway along $p_{F,\la}(w)$, the fiber $L_u$ is displaceable.
\end{lemma}

\begin{proof}  Let $\Phi:M\to\De$ be the moment map of the toric manifold corresponding to $\De$, and consider $\Phi^{-1}(p)$ where $p:=p_{F,\la}(w)$.
We may choose coordinates on $\R^n\supset \De$
 so that $F = \{x_1=0\}$ and $\la = (1,0,\dots,0)$.  
 Formula \eqref{eq:Rn} implies that there is a corresponding Darboux 
chart   on $M$  with coordinates $z_1,\dots,z_n$ such that
$$
\Phi^{-1}(p) = \bigl\{z\,:\, |z_1|\le a, |z_i|= \mbox{ const}\bigr\},
$$
where  $a$ is the affine length of the probe  $p$.
 Hence
there is a diffeomorphism from $\Phi^{-1}(p)$  to $D^2(a)\times T^{n-1}$ that takes the restriction of the symplectic form to $\pr^*(dx\wedge dy)$ where $\pr:
\R^2\times T^{n-1}\to \R^2$ is the projection and $D^2(a)$ is the disc with center $0$ and area $a$.  Further this diffeomorphism takes $L_u$ to $\p D^2(b)\times T^{n-1}$ where $b=d_{\aff}(w,u)$.
But when $b<a/2$ one can displace the circle $\p D^2(b)$ in $D^2(a)$ by a compactly supported  area preserving isotopy.  Therefore $L_u$ can be displaced inside $\Phi^{-1}(p)$ by an isotopy that preserves the restriction of $\om$.  But this extends
 to a Hamiltonian isotopy of $M$ that displaces  $L_u$.
\end{proof}

\begin{defn}  Let $\De$ be any rational polytope and $u\in \intt\De$. If there is a probe $p_{F,\la}(w)$  through $u$ that satisfies the conditions in Lemma \ref{le:probe} we say that $u$ is {\bf displaceable by
the probe} $p_{F,\la}(w)$.
\end{defn}

\subsection{The point $v_0$.}\labell{ss:v0}

In \cite{FOOO1},  Fukaya, Oh, Ohta and 
Ono construct a  point $v_0$ in $\De$ by the 
following procedure. For $u\in\De$,
let $s_1(u): =\inf\{\ell_i(u):1\le i\le N\}$ where $\ell_i(u)$ is as in equation (\ref{eq:li}).  Let $P_0:=\De$ and $I_0: = \{1,\dots,N\}$ and define
\begin{eqnarray*}
S_1:&=&\sup\,\{s_1(u):u\in P_0\},\\
 P_1:
&=&\{u\in P_0:s_1(u)=S_1\},\\
I_1:&=& \{i\in I_0: \ell_i(u) = S_1 \mbox{ for all } u\in P_1\}.
\end{eqnarray*}
Then  $\De\less P_1 = \{u\in\De: \exists j\in I_0, \;\ell_j(u)<S_1\},$
and $P_1 =  \{u\in\De: \ell_j(u)\ge S_1\; \forall j\in I_0\}$.  
It follows from the definition of $S_1$ that $P_1$ is nonempty and it
 is easy to check that it is convex.
If the plane $\ell_i = S_1$ intersects $\intt P_1$ but does not contain it, there will be points $u\in P_1$ with $\ell_i(u)<S_1$ which is impossible. 
Therefore for each $i\in I_0$,
the function $\ell_i(u)$ is either equal to $S^1$ on $P_1$ or strictly greater than $S_1$ on $P_1$.   In other words, 
$$
I_1 = \{i: \ell_i(u) = S_1 \mbox{ for some } u\in \intt P_1\}.
$$
It follows easily that $|I_1|\ge 2$ (since if $I_1=\{i\}$ one can 
increase $s_1(u)$ by moving off $P_1$ along the  direction $-\eta_i$.) Hence $\dim P_1<n$. 
 
 
As an example, observe that if $\De$ is a rectangle with sides of lengths $a<b$, then $P_1$ is a line segment of length $b-a$.
Further, in the monotone case, we show at the beginning of \S3 that 
one can  choose coordinates so that $\De$ contains $0$ and is given by equations of the form \eqref{eq:De} with all $\ka_i = 1$.    Then, $\ell_i(0) = 1$ for all $i$.  Moreover,  
  any other point $y$ of
$\intt \De$  lies on a ray from $0$ that exits through some facet 
$F_j$.  Hence 
$\ell_j(y) <1$.  Thus
 $S_1=1$ and $P_1=\{0\}$.

Inductively, if $\dim P_k>0$, define $s_{k+1}:P_k\to\R,S_{k+1},P_{k+1}$ and $I_{k+1}$  by setting
\begin{eqnarray*}
s_{k+1}(u): &=&\left\{\begin{array}{ll}\inf\{\ell_i(u): \ell_i(u)>S_k\} &\mbox{if } u\in \intt P_k,\\
S_k&\mbox{if } u\in \p P_k,\end{array},\right.\\
S_{k+1}: &=&\sup\{s_{k+1}(u): u\in P_k\},\\
 P_{k+1}: &=&
\{u\in P_k:s_{k+1}(u)=S_{k+1}\},\\
I_{k+1}:&=& \{i: \ell_i(u) = S_{k+1}, \mbox{ for all } u\in P_{k+1}\}.
\end{eqnarray*}
Arguing much as above, they show in \cite[Proposition~9.1]{FOOO1}
 that $s_{k+1}$ is a continuous, convex piecewise affine function, that
  $P_{k+1}$ is a nonempty  convex polytope lying in $\intt P_{k}$,
and that
$$
I_{k+1}=  \{i: \ell_i(u) = S_{k+1} \mbox{ for some } u\in \intt P_{k+1}\}.
$$
It is not hard to see that   $\dim P_{k+1} <\dim P_k$ unless 
the functions $\ell_i, i\in I_{k+1},$ are constant on $P_k$.    Because $\De$ is bounded,  
at least one of the functions  $\ell_j$ for  $j\notin \cup_{r\le k} I_r$ must be nonconstant on $P_k$ when  $\dim P_k> \{0\}$.  Therefore,  after a finite number $s$ of steps one must have $\dim P_{k+s}<\dim P_k$.    Hence there is $K\le N$ such that  $P_K$ is a point; call it $v_0$.  By   \cite[Theorem~1.4]{FOOO1},
  $HF^*(L_{v_0},\chi)\ne 0$ for suitable $\chi$ when $\De$ is rational.
Observe also  that
\begin{equation}\labell{eq:Sk}
S_1 < S_2 < \dots< S_K,
\mbox{ and } I_k =\{i:\ell_i(v_0)=S_k\}.
\end{equation}
Further, $\ell_j(v_0)>S_K$ for all $j\notin \mbox {some } I_k$.
Finally observe that if  $V(P_k)$ denotes the  plane spanned by the vectors lying in $P_k$ for some $k\ge 0$, the fact that $P_{k+1}$ lies in the interior of  $P_k$ implies that\SSS

\begin{quote}
(*)  {\it  if $\dim V(P_{k+1}) <\dim V(P_k)$, 
the normals $\eta_i, i\in I_{k+1},$ project to vectors $\ov\eta_i$ in 
$V(P_k)/V(P_{k+1})$ whose nonnegative combinations $\sum q_i \ov\eta_i, q_i\ge 0,$ span  this quotient space.}
\end{quote}

\NI
(Really one should think of the normals $\eta_i$ as lying in the dual space to $\R^n$ so that this projection is obtained by restricting the 
linear functional $\langle\eta_i,\cdot\rangle$.)

\begin{rmk}\labell{rmk:FO}\rm It is claimed in early versions of 
\cite[Prop.~9.1]{FOOO1} that $\dim P_{k+1}<\dim P_k$ for all $k$.  But this need not be the case.
For example, suppose that $\De$ is the product of $\De'$ with a long interval, where $\De'$ is a square with one corner blown up a little bit as in Figure \ref{fig:sq}.  Then $I_1$ consists of the labels of the four facets of the square, $I_2$ contains just the label of the exceptional divisor, while $I_3$ contains the two facets at the ends of the long tube.  Correspondingly, $P_1$ is an interval, $P_2$ is a subinterval of $P_1$ and $P_3=\{v_0\}$ is a point.
\end{rmk}

\begin{figure}[htbp] 
   \centering
 \includegraphics[width=4in]{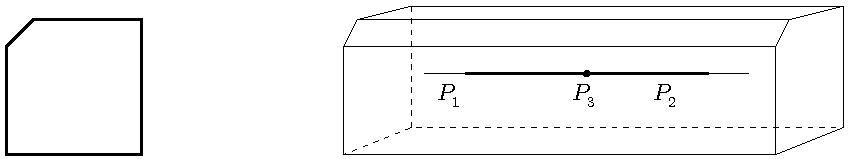} 
   \caption{The construction of $v_0$ for the polytope considered in Remark~\ref{rmk:FO}.}
   \label{fig:sq}
\end{figure}

\begin{lemma}\labell{le:v0}  For every rational polytope $\De$ the point $v_0$  is not displaceable by a probe.
\end{lemma}

\begin{proof}   Suppose that $v_0$ is displaced by a probe 
$p=p_{\la}(w_1)$ 
that enters $\De$ at the point $w_1\in \intt F_1$ and exits  $\De$ through $w_2\in F_2$.    Then 
$$
\ell_1(v_0)= d_{\aff}(w_1,v_0) < d_{\aff}(v_0,w_2)\le \ell_2(v_0),
$$
where the last inequality follows from Lemma ~\ref{le:affd}.

Recall from equation (\ref{eq:aff}) that  $d_{\aff}(v_0,w_2)$ is just the affine distance $d_{\la}(v_0,F_2)$ of $v_0$ from 
 $F_2$ in direction $\la$.  
Because the ray from $v_0$ to $w_2$ in direction $\la$  goes through no facets of $\De$ until it meets  $F_2$ (and perhaps  some other facets as well)  at $w_2$, we have
$$
\ell_1(v_0) < d_\la(v_0,F_2)\le d_{\la}(v_0,F_i),\quad \mbox{ for all } i.
$$
A similar argument applied to the ray from $v_0$ in direction $-\la$ gives
$$
\ell_1(v_0) < d_{-\la}(v_0, F_i) \quad \mbox{ for all } i\ne 1.
$$
(Here we use the fact that $w_1\in \intt F_1$ so that the ray meets $F_1$ before all other facets.)
 But if  $d_{\la}(v_0,F_i)<\infty$ then
 $d_{\la}(v_0,F_i)\le \ell_i(v_0)$ by Lemma \ref{le:affd}.
 Therefore, 
 for all facets $F_i, i\ne 1,$ that are not parallel to $\la$, we have
 \begin{equation}\labell{eq:lv0}
 \ell_i(v_0)> \ell_1(v_0).
\end{equation}

Now observe that because $P_K$ is a single  point $\{v_0\}$  the vectors $\eta_i, i\in I_k, 1\le k\le K,$ span $\R^n$.   Therefore there is some $k\le K$ such that the $F_i, i\in I_k, $ are not all  parallel to $\la$.  Let $r$ be the minimum such $k$, and let $j\in I_r$ be such that $F_j$ is not parallel to $\la$.  If $j\ne 1$ then $S_r=\ell_j(v_0)>\ell_1(v_0)$ by 
equation\eqref{eq:lv0}.  Hence equation \eqref{eq:Sk} implies that
$1\in I_k$ for some $k<r$, which is impossible since $\la$ is not parallel to $F_1$.
On the other hand, if $j=1$ the same reasoning shows that all other elements of $I_r$  correspond to facets that are parallel to $\la$.   
Since by hypothesis the same is true for the elements of $I_{r-1}$.
Therefore $\la \in V(P_k)$ for $k<r$ but $\la\notin 
 V(P_r)$, so that $\la$ has nonzero image in $V(P_{r-1})/V(P_r)$.    But because there is only one $i\in I_r$ for which $\ell_i$ varies along $\la$ this contradicts (*).
  \end{proof}

See Example \ref{ex:2blow} for an example that illustrates how the point $v_0$ varies as the facets of $\De$ are moved.

\begin{rmk}\labell{rmk:ghost}\rm Later we need a slight generalization of this argument in which the set of functions $\ell_i, 1\le i\le N,$ that determine the facets of $\De$ are augmented by some other nonconstant linear functions $\ell_j'= \ka_j' - \langle\cdot,\eta_j'\rangle, j\in J,$ that are strictly positive on $\De$. 
Thus the hyperplanes $A_j'$ on which these functions vanish do not intersect $\De$, so that  the functions $\ell_j'$
 correspond to ghost (or empty) facets of $\De$.
But then, for all $v\in \De$, $ i\in \{1,\dots,N\}$ and $j\in J$, we have
 $$
\ell_i(v) = d_{\eta_i}(v,F_i) < d_{\eta_i}(v, A_j') \le 
d_{\eta_j'}(v, A_j') = \ell_j'(v).
$$
Therefore the maximin procedure that constructs $v_0$ is unaffected by the presence of the $\ell_j'$.  Also, the proof of Lemma \ref{le:v0} goes through as before.
\end{rmk}

\subsection{Stable Displaceability}\labell{ss:stdis}

We end this section with a brief digression about stably displaceable fibers.
The following  definitions are taken from Entov--Polter\-ovich \cite{EPrig}.

\begin{defn}\labell{def:stdis}  A point $u\in \intt\De$ of a smooth moment polytope is said to be {\bf stably displaceable} if  $L_{u}\times S^1$ is displaceable in $M_\De\times T^*S^1$ where $S^1$ is identified with the zero section.
Moreover  $L_{u_1}$ (or simply $u_1$) is called   a {\bf stable stem} if all  points in $\intt \De\less u_1$ are {\it stably displaceable}
\end{defn}

Theorem~2.1 of \cite{EPrig} states  that $L_u$ is stably displaceable if there is an integral vector $H\in \ft$  such that the corresponding circle action $\La_H$ satisfies the following conditions:

\begin{itemize}\item the normalized Hamiltonian function $K_H$ that  
generates $\La_H$  does not vanish on $L_u$;

\item $ \La_H$ is {\it compressible}, that is,  when considered as a loop in the  group $\Ham(M_\De,\om)$ 
of Hamiltonian symplectomorphisms, some multiple of the  circle $\La_H$ forms a contractible  loop in  $\Ham(M_\De,\om)$.
\end{itemize} 
\NI
It is easy to check that $K_H:M_\De\to \R$ has the form
$$
K_H(x) = \langle H,\Phi(x)-c_\De\rangle,
$$
where $c_\De$ is the center of gravity of $\De$.  The paper \cite{MT} makes a detailed study of those $ H$  for which $\La_H$ is compressible. 
This condition implies that the quantity  $\langle H,c_\De\rangle$ depends linearly on the positions of the facets of $\De$, and so the corresponding $H$  are called {\it mass linear functions} on $\De$.

There are two cases,   according to whether the circle $\La_{kH}$ contracts in $\Isom(M)$ or only in $\Ham(M)$, where $\Isom(M)$ is the group  of isometries of the
canonical K\"ahler metric on $M: = M_\De$  obtained by thinking of it as a (nondegenerate)\footnote
{
Here $N$ is the number of facets of $\De$, i.e. there are no \lq\lq ghost"
(or empty) facets. With this assumption the K\"ahler structure is unique.}
symplectic quotient  $\C^N/\!/T'$.  In the first case  $H$ is called inessential, while in the second $H$ is essential.  
The inessential case can be completely understood.  The following argument uses the definitions and notation of \cite{MT} without explanation.\footnote
{
The above definition of inessential is equivalent to the one of \cite[Definition~1.14]{MT} by \cite[Corollary~1.28]{MT}.}

\begin{prop}\labell{prop:stdis} The fiber $L_u$ is stably displaceable by an inessential $H$ if and only if it may be displaced by a probe $p_{F,\la}(x)$ whose direction vector $\la$ is parallel to all but two of the facets of $\De$, namely the entering and exiting facets of the probe.  
\end{prop}  
\begin{proof}   Suppose first that $L_u$ is displaceable by a probe $p_{F,\la}(w)$ with the given property.  Then, by \cite[Lemma~3.4]{MT}, the entering and exiting facets $F: = F_1$ and $ F_2$ of the probe are equivalent and 
there is  an affine reflection of $\De$ that interchanges them.
(Cf. \cite[Definition~1.12]{MT}.)  Moreover, 
$\la$ must be integrally transverse to the exiting facet $F_2$.  Hence the
hyperplane that is fixed by this symmetry contains the midpoint of the probe 
as well as the center of gravity $c_\De$.
 Hence, if $H =\eta_1-\eta_2$, 
$K_H$ does not vanish on $L_u$.  Moreover 
 $\La_H$ is compressible by \cite[Corollary~1.28]{MT}.
Thus $u$ is stably displaceable by an inessential $H$. 

Conversely, if $u\in \intt\De$ is stably displaceable, there is an inessential $H$ 
such that $K_H(L_u):=\langle H,u-c_\De\rangle  \ne 0$.  
	Then \cite[Corollary~1.28]{MT} 
 implies that $H=\sum \be_i\eta_i$ where
$\sum_{i\in I}\be_i=0$ for each equivalence class of facets $I$.  
But each such $H$ is a linear combination of (inessential) vectors $H_\al$ 
of the form $\eta_{\al_2}-\eta_{\al_1}$ where $\al_1,\al_2$ are equivalent.  Therefore there is some pair $\al$ such that  $\langle H_\al,u- c_\De\rangle < 0$.  Let $p$ be the probe from $F_{\al_1}$ through $u$ in direction $\la = H_\al$. (Observe that $\la$ does point into $\De$ since the $\eta_i$ are outward normals.) Then the probe must start at some point in $\intt F_{\al_1}$ since it is parallel to all facets that meet $F_{\al_1}$ and $u\in \intt\De$.    Moreover,
because there is an affine symmetry that interchanges the facets $F_{\al_1}, F_{\al_2}$ while fixing the others, $c_\De$ must lie half way along this probe. Hence, because 
 $\langle H_\al,u\rangle < \langle H_\al,c_\De\rangle $ this probe displaces $u$. \end{proof}

The geometric picture 
for fibers stably displaceable by an essential mass linear $H$ is much less clear.  We show in \cite{MT} that there are no monotone polytopes in 
dimensions $\le 3$ with essential $H$.  In fact,  
\cite[Theorem~1.4]{MT} states that there is exactly one family $Y_a(\ka)$  of 
 $3$-dimensional polytopes with essential $H$. They correspond to
 nontrivial bundles over $S^2$ with fiber $\C P^2$, and always have a symplectically embedded $2$-sphere which is a section of the bundle and lies
 in a class $A$ with $c_1(A)< 0$.   Hence they cannot be monotone.
  (In \cite[Example~1.1]{MT}, this section is represented by the shortest vertical edge, which has Chern class $2-a_1-a_2$ where $a_1,a_2\ge 1$ and $a_1\ne a_2$.)

 It is not clear whether there are
higher dimensional monotone polytopes   with essential $H$.  In particular,
 at the moment there are no  examples of monotone polytopes for which $u_0$ is known to be a stable stem but not known to be a stem.

\section{Monotone polytopes}

There are several possible definitions of a monotone (moment) polytope.
We have chosen to use one that is very closely connected to the geometry of $\De$.

\begin{defn}\labell{def:mono}
 We shall call a  simple smooth polytope $\De$  {\bf monotone} if:\SSS
 
\NI $\bullet$
$\De$ is an integral  (or lattice) polytope in $\R^n$ 
with a unique interior integral point  $u_0$,\SSS


\NI $\bullet$  $\De$ satisfies the  {\bf vertex-Fano} condition:
for each vertex 
$v_j$ we have
$$
v_j+\sum_ie_{ij} = u_0,
$$ 
where  $e_{ij}, 1\le i\le n,$ 
are the primitive integral
 vectors from $v_j$ pointing along the edges of $\De$.
\end{defn}
 
 It follows that for every vertex  $v$ one can choose coordinates for which $u_0=(0,\dots,0)$, $v=(-1,\dots,-1)$  and the facets through $v$ are $\{x_i=-1\}$, $i=1,\dots,n$.  In particular $\ell_j(u_0) = 1$ for all facets $F_j$.   Thus if we translate $\De$ so that $u_0=\{0\}$ the structure 
 constants $\ka_i$ in the formula (\ref{eq:De}) are all equal to $1$.
 
 \begin{rmk}\labell{rmk:inter}\rm  (i) 
 An equivalent formulation is that $\De$ is a simple smooth lattice  polytope with $\{0\}$ in its interior and such that the 
 structure 
 constants $\ka_i$  are all equal to $1$.  To see this, note that if $v$ is a vertex and $e_i$ are the primitive integral vectors along the edges from $v$
 then the lattice  points in $\De$ may all be written as
  $v + m_ie_i$ for some non-negative integers $m_i\ge 0$.  
  Thus $0$ has such an expression, and in this case the $m_i$ are  just the structural constants.  Thus our definition is equivalent to 
  the usual definition of Fano for  the   dual polytope 
$\De^*$ (the simplicial polytope
determined by the fan of $\De$).\MS

\NI (ii)
  Although it is customary to assume that the point $u_0$ is the {\it unique} interior integral point, it is not necessary to do this. For if
 we assume only that $u_0 \in \intt \De$ and that
 the vertex-Fano condition is satisfied by   every vertex 
 we may conclude as above that $\ell_j(u_0) = 1$ for all facets $F_j$.
 Therefore there cannot be another 
 integral interior point $u_1$.  For in this case, we must have
$d_{\aff}(u_0,y) > 1$ where  $y\in F$ is the point 
where the ray from $u_0$ through $u_1$ exits $\De$.  But by
 Lemma \ref{le:affd} we must also 
 have $d_{\aff}(u_0,y) \le |\ell_F(u_0)| = 1$, a contradiction.
  \end{rmk}

 It is  well known that the monotone condition for moment
 polytopes is equivalent to the condition that the corresponding
 symplectic toric manifold $(M_\De,\om_\De)$
 is monotone in the sense that $c_1: = c_1(M) = [\om_\De]$.  A proof is given in \cite[Proposition~1.8]{EPrig}.
We include  another for completeness.   In the statement below we denote the moment map by $\Phi:M\to \De$.  Recall also that by construction the affine length of an edge $\vareps$ of $\De$ is just 
$\int_{\Phi^{-1}(\vareps)}\om_\De$.

\begin{lemma}\labell{le:Fa} Let $\De$ be a smooth integral moment polytope with an interior integral point $u_0$.
Then $\De$ is monotone  if and only if the affine length of each edge 
$\vareps$ of $\De$ equals $c_1(\Phi^{-1}(\vareps))$.
 \end{lemma}
 
 \begin{proof}   Suppose that $v_j+\sum_ie_{ij} = u_0$ for all vertices.
  Suppose that $\vareps=:a e_{01}$ is the edge between the vertices $w_0$ and $w_1$, and assume  that the other edges $\vareps_{0i}$ starting at $w_0$ and in the directions $e_{0i}$ end at the points $w_i,2\le i\le n$.
   Move $\De$ by an integral linear transformation so that $w_0=(0,\dots,0)$ and 
   so that $e_{0i}$ points along  the $i$th coordinate direction, for $i=1,\dots,n$.
   Then $w_1=(a,0,\dots,0)$ and we need to check that
   $c_1(\Phi^{-1}(\vareps_{01})) = a$.  Note that in this coordinate
    system $u_0=(1,\dots,1).$

   Consider the vertices $y_1=w_0, y_2,\dots,y_n$ connected to $w_1$.
   There is one such vertex $y_j=w_1+m_je'_{1j}$ in each 
   of the $2$-faces $f_{01j}=
   \sspan(e_{01},e_{0j})$, $j>1$, containing $e_{01}$.  (Here $e'_{1j}$ is a primitive integral vector pointing from $w_1$ to $y_j$.)   Therefore 
   $e'_{1j} = (b_j,0,\dots,0, 1,0,\dots,0)$,
   where the $1$ appears as the  $j$th component.\footnote
   {
   The $j$th component must be $1$ because the $e'_{1j}, j>1,$ together with $-e_{01}$ form a lattice basis.}
   Therefore the identity
 $(1,\dots,1)= (a,0,\dots,0) + (-1,0,\dots,0) +\sum_j e'_{1j}$ implies that
  $$
 1= a -1 +\sum_{j\ge 2} b_j.
 $$
   Now consider the $S^1$ action on $M_{\De}$ with Hamiltonian
  given by $pr_1\circ\Phi$, where $pr_1$ denotes  projection to the first coordinate.  The weights of this action at $\Phi^{-1}(w_1)$ are 
  $(-1,b_{2},\dots,b_{n})$ with sum $m_1= -1+\sum b_{j}$,
  while its weights at $\Phi^{-1}(w_0)$ are $(1,0\dots,0)$ with sum $m_0=1$.  
  Therefore 
  $$
  c_1\bigl(\Phi^{-1}(e_{01})\bigr) = m_0-m_1=1-(-1+\sum_j b_{j}) = a,
  $$
as required.  The proof of the converse is similar.
\end{proof}

In the next lemma  we denote by $\Ss:=\Ss(\De): = \De\cap (-\De)\cap \bigr(\Z^n\less \{0\}\bigl)$ the set of nonzero 
symmetric integral points of $\De$, where we assume that  $u_0=\{0\}$.

\begin{lemma}\labell{le:probO}  Let $\De$ be a monotone polytope.
If $U$ is a sufficiently small  neighborhood of
 $u_0=\{0\}$, then the set of direction vectors of the probes that displace some point in $U$ is precisely $\Ss$.
\end{lemma}
\begin{proof}   
Given $U$, let   $\La(U)$ be the set of direction vectors  of probes $p_{F,\la}(w)$ that displace some point $y$ in $U$.

 We first claim  that $\Ss\subset \La(U)$ for all $U$.  To see this,  
 observe first that if $\la\in \Ss$ then $\la$ (considered as a direction vector)  is integrally transverse to every facet $F$ containing the point $-\la$. (This holds because we may choose coordinates 
 so that $u_0=\{0\}$ and $F= \{x_1=-1\}$.)  Therefore  
 for each such pair $\la,F$ there is a probe $p_{F,\la}(-\la)$.
This exits  $\De$ at the point $\la$ and  has  midpoint at $\{0\}$.  If $-\la\in \intt F$ this probe  displaces all points less than half way along it.  
 Moreover, if $\la \not\in \intt F$, then because $U$ is open any probe $p_{F,\la}(w)$ starting at a point $w\in \intt F$  sufficiently close to $-\la$ will displace some points of $U$.  Hence  $\Ss\subset \La(U)$ as claimed.

 We next claim that if $\la\in \La(U)$ and
  $U$ is sufficiently small then  $\pm\la\in \De$.  Since $\la\in \Z^n$ this means that $\la\in \Ss$, which will complete the proof.
  
 To prove the claim, consider a probe $p_{F,\la}(w)$ that displaces some $y\in U$, and choose coordinates $x_1,\dots, x_n$ on $\R^n$ so that $F \subset \{x_n=-1\}$.    Then the direction $\la$  is an integral vector with last coordinate $=1$.  Therefore $-\la$ is an integral point in the plane $\{x_n=-1\}$.
To arrange that $-\la\in F$,
assume that in this coordinate system $U$ is contained in the Euclidean ball about $\{0\}$ with radius $\eps$.  Then  if 
 $y=(y_1,\dots,y_n)\in U$ is displaced by $p_{F,\la}(w)$ we must have $ y-(1-y_n)\la = w\in F$.    Therefore the Euclidean distance of $-(1-y_n)\la$ to $F$ is  at most $\eps$.  Since $|y_n|<\eps$ and $-\la$ is integral, this implies that $-\la\in F$ if $\eps$ is sufficiently small.   Similarly, because  $y$ is less than half way along the probe, $y + (1-y_n)\la\in \De$.  
 Therefore the Euclidean distance of $(1-y_n)\la$ to $\De$ is  at most $\eps$, and so, by the integrality of $\la$ we may assume that $\eps$ is so small that $\la\in \De$ also.
 
The permissible size of
 $\eps $ here depends only on the image of $\De$ in our chosen coordinate system.  But we need make at most one such choice of coordinate system for each facet.
 Hence
  we may choose $\eps>0$ so small
 that the above argument works for all $\la\in \La(U)$.
\end{proof}

\subsection {Probes and the Ewald conjecture}\labell{ss:starE}

The (dual version of the) {\bf Ewald conjecture} of \cite{Ew} claims that if $\De$ is a monotone polytope then the set
$\Ss(\De)$ of integral symmetric points 
contains an integral basis of $\R^n$. Essentially nothing is known about the truth of this conjecture in general; for example, it  is even not known  whether $\Ss(\De)$ is  nonempty.  However, the conjecture has been checked by {\O}bro \cite{Ob} in 
dimensions  $\le 8$.  Moreover,  {\O}bro observes that in these dimensions a stronger form
 of the Ewald conjecture holds.  Namely
 {\it in dimensions $\le 8$  for every facet $F$, $\Ss(\De)\cap F$ contains an integral basis for $\R^n$.}
 
To prove displaceability by probes one needs a slightly different condition.  Given a face $f=\cap_{i\in I}F_i$ we shall denote by $\Starr (f)$ the union $ \cup_{i\in I}F_i$ of the facets containing $f$ and by $\starr (f)$ 
the union $ \cup_{i,j\in I,i\ne j}F_i\cap F_j$ of the codimension $2$  faces containing $f$.  
Further we define the deleted star
$\istar(f)$  as:
$$
\istar(f): =\Starr(f)\less\starr(f) =  \bigcup_{i\in I}F_i \;\less \;\bigcup_{i\ne j, i,j\in I}F_i\cap F_j.
$$
  In particular, $\Starr (F) = F = \istar(F)$ for any facet  $F$.

\begin{defn}\labell{def:starE} Let $\De$ be any smooth polytope with $\{0\}$ in its interior. We will say that  $\De$ satisfies the {\bf weak Ewald condition} if 
$\Ss(\De)$ contains an integral basis of $\R^n$, and that
it satisfies the {\bf strong Ewald condition} if 
$\Ss(\De)\cap F$ contains an integral basis of 
$\R^n$  for every facet $F$.
A face $f$ satisfies the
{\bf star Ewald condition} if
 there is some element $\la\in \Ss(\De)$ with $\la\in \istar (f)$ but $-\la\not\in \Starr (f)$.   Further,
$\De$ satisfies the
star Ewald condition (or, more succinctly, is {\bf star Ewald}) if all its faces have this property.  
\end{defn}

\begin{rmk}\labell{rmk:star}
\rm (i) Because $\la$ and $-\la$ cannot lie in the same facet $F$, the star Ewald condition is satisfied by any facet $F$  for which  $\Ss\cap F\ne \emptyset$.   
\SSS

 \NI (ii)  If $\De$ is monotone then, because it has a unique interior integral point $u_0$,    we must have $u_0=\{0\}$ in the above definition.\SSS
 
 \NI (iii)  The star Ewald condition makes sense for 
 any (not necessarily smooth) polytope containing $\{0\}$ in its interior, and in particular for reflexive polytopes.  These are integral polytopes such that 
 $\{0\}$ has affine distance $1$ from all facets.  Thus, as in the monotone case,  the special point $v_0 = \{0\}= P_1$ is reached at the first step of the
 maximin construction in \S2.  However, we shall not work in this generality
 because we are interested in the question of when there is a unique nondisplaceable point, and the examples in Remark \ref{rmk:refl} suggest that this happens only in the smooth case.
\end{rmk}


The relationships between the strong Ewald and star Ewald   conditions  are
 not completely clear.  However, as we see in the
 next lemma, the star Ewald condition does imply the weak Ewald condition for monotone polytopes.\footnote{
I am indebted to  Benjamin Nill for  sharpening my original claim.} 

\begin{lemma}\labell{le:span}  If  a monotone
 polytope $\De$ has a
 vertex $v$ such that every face containing $v$ is star Ewald, then
$\De$ satisfies the weak Ewald condition.
\end{lemma}

\begin{proof}  Choose coordinates for $\De$
so that $v = (-1,\dots. -1)$ and the facets through $v$ are 
$\{x_i=-1\}$.   Then $\De$ lies in the quadrant $\{x_i\ge -1\}$, so that the coordinates of  any point $\la \in \Ss(\De)$ must lie in
$\{0,\pm 1\}$.
By the star Ewald condition for the $0$-dimensional face $v$
(and renumbering the coordinates if necessary)
we may assume that there exists some $\lambda \in\Ss(\De)$ 
with $\lambda_1 = -1, \lambda_2 = ... = \lambda_n = 0.$
Now consider $f = \{x_2 = ... = x_n = -1\}$. Again,
by the star Ewald condition for $f$ (and renumbering  if necessary)
 we find
that there is $\lambda' \in\Ss(\De)$ with $\lambda'_2 = -1,
\lambda'_3 = ... = \lambda'_n = 0$. Proceeding in this way
we get $n$ lattice points in $\Ss(\De)$ forming a lattice basis.
\end{proof}

Here is another easy result.

\begin{lemma}  If a facet $F$ of  a monotone polytope  $\De$ contains a lattice basis consisting of  points in $\Ss(\De)$,
then each of its codimension $2$ faces satisfies the  star Ewald condition.
\end{lemma}

\begin{proof}   With coordinates as in the previous lemma, 
it suffices to consider a face $f= F_1\cap F_2 = \{x_1=x_2=-1\}$ such that $\Ss(\De)\cap F_1$ contains a lattice basis.
Since $\Starr f = F_1\cup F_2$ and $\starr f = f$,
we need to show that there is a symmetric point $v$ in 
$F_1\less f$  with $-v\notin F_1\cup F_2$.

By assumption the points in $\Ss(\De)\cap F_1$ form a lattice basis.
 If some point in this set has the form  $v_1 = (-1,0,y_3\dots,y_n)$ then we are done.  Otherwise, there is a lattice basis consisting of
points  $v_1,\dots, v_n$ that all have  second coordinate $y_2 = \pm 1$. 
The points $v_1,v_j\pm v_1, j\ge 2$, also form a lattice basis, and we may choose the signs so that the first coordinate of each $v_j\pm v_1, j\ge 2,$ is zero.
But then the second coordinates of these points
are always multiples of $2$, which is impossible, since they form a matrix of determinant $\pm 1$.
\end{proof}

We now prove Theorem \ref{thm:stEmain}, which states
that for monotone polytopes $\De$ the  star Ewald condition
is equivalent to the property that
 every point in $\intt\De\less\{0\}$ can be displaced by a probe. 
\MS

\begin{proof}[Proof of Theorem \ref{thm:stEmain}]\,\,  For each point $x\in \De$ and disjoint face $f$ denote 
by $C(f,x)$
the (relative) interior of the cone with vertex $x$ and base $f$. Thus
$$
C(f,x) = \{rx + (1-r)y : r\in (0,1), y\in \intt f\}.
$$
(Here, the  relative interior $\intt f$ is assumed to have the same dimension as $f$.  In particular, for every vertex $v$, we have $v=\intt v$.)
Thus $$
\intt\De\less \{0\} = \cup_f C(f,\{0\}).
$$

If $p_\la(-\la)$ is a probe through $\{0\}$ starting at 
$-\la\in F\cap \Ss$, then $\la\in \De\less F$ so that
by convexity
$C(F,\la)\subset \intt \De$. It is then easy to check that
 all points in $C(F,\{0\})$ are displaceable by the probes 
$p_\la(x), x\in \intt F,$ in this direction $\la$.  

More generally, for each face $f$  of $\De$, the points 
in $C(f,\{0\})$ are displaced by  probes in the direction $\la\in \Ss$ if $-\la\in \istar (f)$ while $\la\notin \Starr (f)$.
For if $-\la\in  F\cap \istar(f)$  then 
$$
W: = C(f,-\la)\subset \intt F,\quad\mbox{ and }\;\;\;C\bigr(W,\la\bigr)\subset\intt\De.
$$
(Here, by slight abuse of notation, we allow the base of our cone to be a subset of a face rather than the face itself.)
Therefore  we may displace the points in $C(f,\{0\})$ by  the probes
$p_\la(w)$ where $w\in W=C(f,-\la)\subset \intt F$.

Conversely, let $f$ be a face such that every point in $C(f,\{0\})$ can be displaced by a probe.  We will show that $f$ satisfies 
the star Ewald condition.  To this end choose coordinates on $\De$ so that
$\De$ has facets $F_i: = \{x_i=-1\}, 1\le i\le n,$ where
$f=\cap_{1\le i\le d}F_i$. 

If  $f=F_1$, then for
$t>0$ consider the slice $\De_t:= \De\cap\{x_1=-t\}$. 
Because $\De$ is integral, there  are no vertices in the slice $\{0>x_1>-1\}$.
Therefore $\De_t$
 is a smooth polytope for $0<t<1$ with facets $F_j\cap \De_t, j\in J,$
where  $J=\{j : 1<j\le N, F_j\cap F_1\ne \emptyset\}$.
Every probe in $\De_t$ is a probe in $\De$.   Therefore,  
by Lemma~\ref{le:v0} there is a point $v_t\in \De_t$ 
that cannot be displaced by any probe in $\De_t$.
  Hence the direction vector
$\la=(\la_1,\dots,\la_n)$  of any probe that displaces $v_t$ must have 
$\la_1\ne0$. 
Now observe that because  $\{0\}$ is the unique point 
with $\ell_i=1$ for all $i$,
the construction of the special point in
 Section \ref{ss:v0} implies that $v_t\to \{0\}$ as $t\to 0$. 
Therefore $v_t$ is in the neighborhood $ U$ of Lemma \ref{le:probO}
 for sufficiently small $t$ so that $\la\in \Ss(\De)$.
If $\la_1 > 0$ then the probe must originate from a point in 
$F_1$.  Letting $t\to 0$ we see that $-\la\in F_1$.
On the other hand, if $\la_1 < 0$ a similar argument shows that $\la\in F_1$.
Thus in both cases $\Ss\cap F_1\ne 0$, as required by the 
star condition for $f=F_1$.

Now suppose that $\dim f=n-d < n-1$ and
let  $\La(f)$ be the set of directions $\la$
 of probes $p_\la(w)$ that displace points of
$C(f,\{0\})$ arbitrarily 
close to $\{0\}$.    
For each  $\la\in \La(f)$ denote 
$$
U_\la(f) = \{y\in C(f,\{0\})\,:\, y \mbox{ is displaced by a probe in direction }\la\}.
$$
Then  $\{0\}$ is in the closure of each $U_\la(f)$, and
$$
\bigcup_{\la\in \La(f)} U_\la(f)
$$
 contains all points in $C(f,\{0\})$ sufficiently close to $\{0\}$.

Now, for each facet $F$ containing $f$ consider the set
 $$
 W_{\la,F}(f) = \{w\in \intt F\;:\mbox{ the probe } p_\la(w) \mbox { displaces some } y\in U_\la(f)\}.
 $$
 Because each such probe 
$p_\la(w)$ meets $U_\la(f)$ less than half way along its length, 
we must have
 $C(W_{\la,F}(f),\la) \subset \intt \De$.   But this implies that 
$\la\notin F$ for any $F\supset f$, i.e. $\la\notin \Starr (f)$.
 Also $-\la\notin \starr f$, since if it were the initial points 
$w$ of the probes 
would lie in $\starr f$ and not in the interior of a facet, as is required.

It remains to check that there is some $\la\in \La(f)$ such that
 $-\la$ is in one of the facets $F_i$ containing $f$.
For this, it suffices that $-\la_i = -1$ for some $i\le d$.
But because $\la\notin F_i$ we know $-\la_i \in \{-1,0\}$ for these $i$.
And if $\la_i=0$ for all $i$ then $\la$ would be parallel to 
$C(f,\{0\})$, or, if $d=n$, would be equal to $\{0\}$.
Since $\la\in \p\De$, the latter alternative is impossible.
Therefore we may assume that $\dim f = n-d > 0$  and must
 check that there is an 
element in $\La(f)$ that is not parallel to $f$. 

To this end, we adapt the argument given above for facets.  We shall suppose
that the elements in $\La(f)$ are all parallel to $f$ and shall then show that there is a nondisplaceable point in $C(f,0)$.

For fixed $t\in (0,1]$ consider the polytope
$$
f_t: = \De\cap \{x_1=\dots=x_d=-t\},
$$
so that $f=f_1$, and define the set $I$ by
$$
i\in I \Longleftrightarrow\Bigl( F_i\cap \{x_1\in (-1,0)\}\ne \emptyset, \mbox{ and }\ell_i
\mbox{ is not constant on } f\Bigr).
$$
Since the functions $\ell_i, i\in I,$ are nonnegative on $f_t$ and the boundary of $f_t$ is the set where at least one $\ell_i$ vanishes, we may define 
a point $v_t\in f_t$ by applying the maximin construction of \S2 to the restriction of the functions $\ell_i, i\in I,$ to $f_t$.  
 The argument in 
Lemma \ref{le:v0} shows that this point $v_t$ is not displaceable by probes 
in $f_t$.   (The only new element in the situation is that some of the $\ell_i$ may represent ghost facets, i.e. they may not vanish anywhere on $f_t$. But this does not affect any of these arguments; cf. Remark \ref{rmk:ghost}.) 
 
The probes of interest to us have directions $\la\in \La(f)$.  Since these points $\la$ lie in the plane $\{x_1=0\}$  there is $\eps>0$ such that
each $\la$ lies in a facet $F$ of $\De$ that intersects $f_t$ for all $t\in(0,\eps)$. Therefore,
as in  Lemma~\ref{le:probO}, the directions $\la\in \La(f)$ are integrally transverse to the facets of $f_t$ for $t\le \eps$ when considered as probes 
in the plane $\{x_i = -t, 1\le i\le d\}$ containing $f_t$.  Hence   the probes of $\De$ with directions $\la\in \La(f)$ form a subset of the probes in $f_t$ for $t\le \eps$.  Therefore they cannot
displace $v_t$.  It remains to prove:\MS

\NI
{\bf Claim:}  $v_t\in C(f,0)\cap f_t$ {\it when } $t\le \eps$.  

To see this, let $F_j, j\in J_f,$ be the set of facets of $\De$ that intersect but do not contain $f$.   Then $J_f\subset I$ and the facets of $f$ are $f\cap F_j, j\in J_f$.   Now observe that if $\langle x,\eta_j\rangle \le 1$ then
$\langle tx,\eta_j\rangle \le t$, so that 
\begin{equation}\labell{eq:tx}
\ell_j(tx)  = 1-\langle tx,\eta_j\rangle\ge 1-t.
\end{equation}
Applying this with $x=p_0 = (-1,\dots,-1)\in \De$ we see that
 $$
 \ell_i(tp_0) = 1- t\langle \eta_i,\,p_0\rangle\ge 1-t,\quad\mbox{ for all }i\in I.
 $$
Further, because
 we chose coordinates so that
$f = \{x\in \De: x_i = -1, 1\le i\le d\}$, equation (\ref{eq:tx}) implies that
when $t\in (0,1]$ we have
\begin{eqnarray*}
 C(f,0)\cap f_t 
 &=&\bigl\{y:\ell_j(y)\ge 1-t, j\in J_f\bigr\}\cap \bigl\{ y_i = -t, 1\le i\le d\bigr\}.
\end{eqnarray*}
 Therefore the maximum value of the function $s_1(y): =\min_{i\in I} \ell_i(y)$  for $y\in f_t$  is at least $1-t$, and because $J_f\subset I$ it is assumed in $C(f,0)\cap f_t$.  
 Thus $P_1$, and hence also $v_t\in P_1$, lies in $C(f,0)$.
 This proves the claim, and completes the proof of the proposition.
\end{proof}

\begin{cor}  If $\De$ is a $2$-dimensional monotone polytope then
every point in $\intt\De\less \{0\}$ may be displaced by a probe.
\end{cor}
\begin{proof} It suffices to check that the star Ewald condition holds,
which is easy to do in
each of the $5$ cases (a square, or a standard simplex 
 with $i$ corners cut off, where $0\le i\le 3$.)
\end{proof}

For the $3$-dimensional version of this result see Proposition \ref{prop:3}.


\section{Low dimensional cases.}\labell{s:special}

\subsection{The $2$-dimensional case.}\labell{ss:2}

In this section we discuss the properties of arbitrary, not necessarily monotone, $2$-dimensional polytopes.
We begin by showing that  there always is an inaccessible point near any {\it short odd edge}.  Here we say that an edge $\vareps$ is {\it odd} if its self-intersection number\footnote
{
By this we mean the self-intersection number of the $2$-sphere $\Phi^{-1}(\vareps)$
in the corresponding  toric manifold $M^4$.
This is the Chern class of its normal bundle, and equals $k$, where
we assume that  $\vareps$ has outward normal $(0,-1)$ and that its neighbor to the left has conormal $(-1,0)$ and to the right  has conormal $(1,k)$; cf. \cite[\S2.7]{KKP}.}
     is odd and negative,
and that it is {\it short} if its affine length is at most half that of its neighbors.

\begin{figure}[htbp] 
   \centering
   \includegraphics[width=4in]{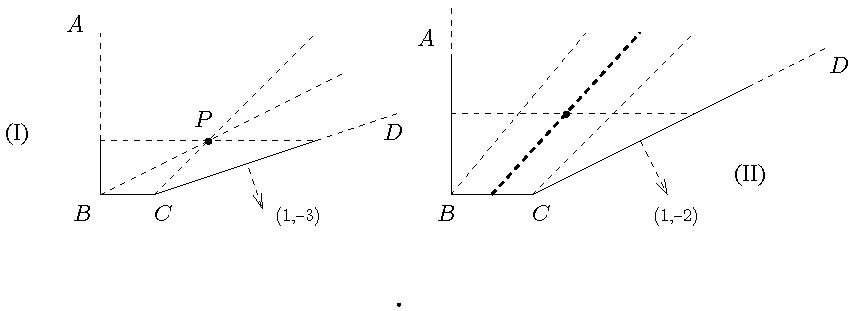} 
   \caption{In (I), $BC$ has self-intersection $-3$. $P$ is the midpoint of the line parallel to $BC$ and a distance $|BC|$ above it. It is not displaceable because the two probes from $BC$  with good (i.e. integrally transverse) direction vectors have initial points at vertices.
Figure (II) illustrates the case when $BC$ has self-intersection $-2$; the heavy line consists of points midway between $AB$ and $DC$.  (This line is integrally transverse to $BC$ because we are in the even case.) Points not on this line can be displaced by horizontal probes, while points on this line that are close to $BC$ can be displaced  by probing  from $BC$ parallel to it.}
   \label{fig:8}
\end{figure}

\begin{lemma}\labell{le:inacc}  Let $A,B,C,D$ be four neighboring vertices on a smooth polygon such that the edge $\vareps=BC$ of affine length $d$  is  short and odd.
Then no probe displaces the midpoint $P$ of the line parallel to $\vareps$ and a distance $d$ above it.
\end{lemma}
\begin{proof} Suppose without loss of generality that $d=1$. Choose coordinates so that $B$ is at the origin and $A,C$ are on the $y,x$-axes respectively as in Figure \ref{fig:8}.  Then the self-intersection condition implies that
the normal to $CD$ is $(1,-(1+2k))$ for some integer $k\ge 0$.  The horizontal distance from $P$ to $BA$ is $k+1$.  Since this is an integer,  the only probes through $P$ that start on $BC$ have initial vertex at $B$ or $C$.
Therefore $P$ cannot be displaced by such probes.  But it also cannot be displaced by probes starting on
$AB$ since these must have direction $(1,a)$ for some $a\in \Z$. 
 By symmetry, the same argument applies to probes from $CD$. Finally note that because $BC$ is short, all probes starting on edges other than $AB, BC$ or $CD$ meet $P$ at least half way along their length
 and so cannot displace $P$.
\end{proof}

Recall from the end of  Section \ref{s:intro}
that $\De(\ka)$ is said to be {\it accessible} if all its points except for $v_0$ are displaceable by probes.

\begin{cor} \labell{cor:inacc} The only smooth
 polygons $\De$ such that $\De(\ka)$  is accessible for all $\ka$ are triangles, and trapezoids with no odd edges.
\end{cor}
\begin{proof} As is illustrated in diagram (II) in Figure \ref{fig:8}, the argument in Lemma~\ref{le:inacc} does not apply to even edges since then the line of midpoints is a good direction from $BC$.  It follows easily that every trapezoid without an odd edge has only one nondisplaceable point.  
  Every other smooth polygon with at least $4$ sides can be obtained by blowup from the triangle or a trapezoid\footnote
  {
  This is well known; see Fulton \cite[\S2.5]{F}, or \cite[Lemma~2.16]{KKP} where the blowup process is called \lq\lq corner chopping".}
   and so has an edge $\vareps$ of self-intersection $-1$, the result of the last blow up.  Clearly, this edge  can be made  short.
\end{proof}

Denote by $\ND_p\subset \De$ the set of points $u\in \intt\De$ that are not displaceable by probes, and by
 $\ND_{HF}\subset \De$ the set of points $u\in \intt\De$ for which $HF_*(L_u,\chi)\ne0$ for some $\chi$.

  \begin{figure}[htbp] 
     \centering
    \includegraphics[width=4in]{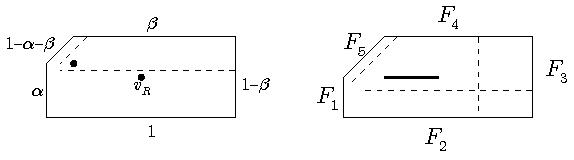} 
     \caption{Some possibilities for $\ND_{HF}$ when $\De$ is a 
     $2$-point blow up of $\C P^2$.  Here  $\ND_p $  is depicted by the dark dots and heavy lines; the dotted lines show permissible directions of probes.}
     \label{fig:probe0}
  \end{figure}
 
\begin{example}\labell{ex:2blow}\rm
 Let $\De$ be the moment polytope of
  a $2$-point blow up of $\C P^2$ as in Figure  \ref{fig:probe0}.
Then the three consecutive edges $F_4, F_5, F_1$ are odd. 
Denote their affine lengths by $L(F_i)$.   We normalize the lengths of the edges of the triangle $T$ formed by $F_2,F_3, F_5$ to be $1$ and denote $\al: = L(F_1), \be: = L(F_4)$, so that
$L(F_5) = 1-\al-\be$.  Without loss of generality, we assume that $\al\le \be$.  We denote by $v_T$ the center of gravity of the triangle $T$ and by $v_R$  the center of gravity of the rectangle $R$
with facets $F_1,\dots, F_4$. 

The first question is: where is $v_0$?
If 
$L(F_5)\le \be$ 
(as in both cases of Figure \ref{fig:probe0}),  then one can check that $v_0 = v_R$.  In this case, one should think of $\De$ as the blow up of
the rectangle $R$.  On the other hand, if 
$L(F_5) = 1-\al-\be > \be = L(F_4) (\ge \al)$ then
$v_R$ can be displaced from $F_5$, and
one should think of $\De$ as the blow up of the triangle $T$.
 If in addition $(\al \le )\be \le \frac 13$, then $v_T$ cannot be displaced from $F_4$ since it is at least as close to $F_2$ as to $F_4$, and it follows that $v_0 = v_T$.   However, if $1-\al-\be > \be>\frac 13$,
 then  $v_T$ can be displaced from $F_4$. 
 One can check in this case that $v_0$ is on the median of $T$ through
 the point $p$ where the prolongations of $F_3$ and $ F_5$ meet, half way between the parallel edges $F_2$ and $F_4$.

Now consider the other points in  $\ND_p$.   We will say that an odd edge is {\it short enough} if 
it is shorter than its odd neighbors and has at most half the length of its even neighbors.  
 Because $\De$ has so few edges, one can check that the statement in
  Lemma \ref{le:inacc} holds for all short enough edges in $\De$.
  
  Throughout the following discussion 
we assume that $\al\le \be$.  As $\al,\be$ vary,
 precisely one of the following cases occurs.
 \MS
 
 \NI (i)  $L(F_5)<L(F_1)$. If $F_5$  is short, then as  in the left hand diagram in Figure \ref{fig:probe0}, $\ND_p$ consists of two points, namely $v_0$ (which coincides with $v_R$) and the point $P$ corresponding to the short edge.  
An analogous statement continues to hold as long as  $L(F_5) < L(F_1)
(\le L(F_4))$, i.e. as long as $F_5$ is short enough: the proof of  Lemma \ref{le:inacc} shows that $P$ cannot be displaced from the facets $F_4, F_5$ or $F_1$ and it cannot be displaced from $F_2$ or $F_3$ because they are too far away. Therefore $\ND_p$ consists of $P$ and $v_0=v_R$, as in the left hand diagram in Figure \ref{fig:probe0}.
 \SSS
 
 \NI (ii) $L(F_5) = L(F_1)< L(F_4)$.  Now there are no short edges and $\ND_p$ is an interval, with $v_0=v_R$ as its \lq\lq middle" end point;
 cf. the right hand diagram in Figure \ref{fig:probe0}.\SSS
 
  \NI (iii) $L(F_5) = L(F_1)= L(F_4) = \frac 13$ (the monotone case). Again there are no short edges; $\ND_p$ is the single point $v_0=v_R=v_T$.\SSS
 
   \NI (iv)  $L(F_4)\ge L(F_5)> L(F_1)$.   Note that
    $L(F_2) = L(F_4) + L(F_5) > 2 L(F_1)$.  Hence,  
  $F_1$ is short enough and  $\ND_p$ consists of $v_0=v_R$ and the point $P$ corresponding to $F_1$.
  \SSS
  
  \NI (v)  $L(F_5)>L(F_4)\ge L(F_1)$. 
  As we saw above, the position of $v_0$ varies depending on the relative sizes of $L(F_4) = \al$ and $\frac 13$.  Further  $F_1$ is always short enough, while $F_4$ may or may not be.  Correspondingly,  $\ND_p$ consists of two or three points.
   \end{example}

  
This example was discussed in  detail in
\cite[Examples~5.7,~10.17,~10.18]{FOOO1} and in \cite[\S5]{FOOO2}, where the authors
showed that $\ND_{HF} = \ND_p$ in all the above cases.
On the other hand, in \cite[Examples~8.2]{FOOO1} the authors 
calculated Floer homology groups in the 
case of  Hirzebruch surfaces and,
in the  case  when the negative curve has self-intersection $-k\le -2$, appear to find
only 
{\it one} point $u\in \ND_{HF}$ (with $4$ corresponding
 deformation parameters $y$).
 In other words, the inaccessible point
 $P$ described in Lemma 4.1 when $k$ is odd is not in $\ND_{HF}$.  
 This seems to be the simplest example where the two sets are different.\footnote
 {
 Of course, this point might be detected by  more elaborate versions of Floer homology.}

It is shown in \cite[\S10]{FOOO1} that if one moves the facets of $\De$ to be in general position (so that the Landau--Ginzburg potential function is nondegenerate),
then $\ND_{HF}$ is finite.  We now show that $\ND_p$ sometimes
 contains an open subset.

\begin{figure}[htbp] 
   \centering
\includegraphics[width=5in]{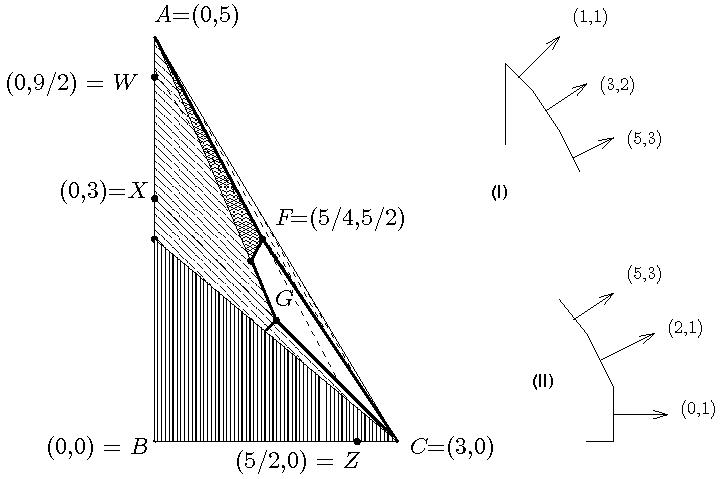} 
   \caption{The shaded regions in the triangle $ ABC$ can be displaced by probes parallel to the shading; the heavy lines and open region cannot be so displaced.  Here $G=(3/2,3/2)$ is the midpoint 
   of $CX$, while $F=(5/4,5/2)$ is the midpoint of $AZ$. 
   Figures (I) and (II) show the cuts needed to smooth the vertices $A$ and $C$.}
   \label{fig:2}
\end{figure}

\begin{lemma}\labell{le:open}  There is a $2$-dimensional smooth polytope with an open set of points that are not displaceable by probes.
\end{lemma}

\begin{proof} The triangle $ABC$ in Figure \ref{fig:2} has vertices $A=(0,5), B=(0,0)$ and $C=(3,0)$, and so is not smooth.
  The points inside the triangles $ABG$ and $CBG$ can be displaced by probes in the directions $\pm(-1,1)$, and all but  a short segment of $BG$ can be displaced 
 by vertical probes from $BC$.  Also, probes from $AB$ in direction $(1,-2)$ displace the points in $ABF$.
   On the other hand,
the best probes from $AC$ are either parallel to $AZ$ in the direction $(1,-2)$ or are parallel to $CW$ in the direction $(-2,3)$.  
(In fact, the latter set of probes displaces no new points.) 
Therefore this triangle contains an open region that cannot be reached by probes.

To get a smooth example, blow up at the vertices $A$ and $C$ along the directions indicated in figures (I) and (II).  Probes starting from these new edges will reach some more points, but these probes must be in one of a finite number of directions.
 (For example, from the edge near $C$ with normal $(1,0)$ one reaches some new points by  probes  in the direction  $(-1,1)$.)  Hence, 
since the new edges can be arbitrarily short, the newly accessible regions can have arbitrarily small area. 
\end{proof}

We leave it up to the reader to construct similar examples in higher dimensions.    Note that the  reason why  one gets an open set of nondisplaceable points is that in the above example  most probes exit through facets that are not integrally transverse to  the direction of the probe.  For example, in the triangle above the horizontal probes from $AB$ exit through $AC$ which is not integrally transverse to $(1,0)$.   Figure \ref{fig:2b} illustrates two more possibilities.

\begin{figure}[htbp] 
   \centering
\includegraphics[width=5in]{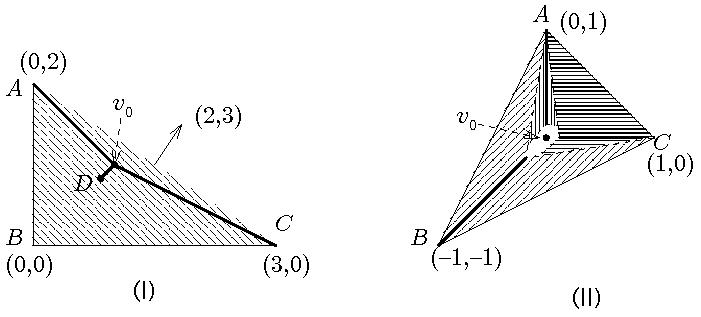} 
   \caption{The shaded regions can be displaced by probes parallel to the shading.  In (I),
 $\ND_p$ consists of  the (open) heavy lines and the points $v_0=(1,1)$ and $D=(\frac67,\frac67)$. 
Points below $D$  on the ray  $v_0D$ can be displaced horizontally from $AB$.  
In (II), $\ND_p$ is a hexagon containing $v_0=(0,0)$ together with parts of the lines $x=y$, $x=0$ and $y=0$. }
   \label{fig:2b}
\end{figure}

\begin{rmk}\label{rmk:FH}\rm   (i) To see why the two sets $\ND_{HF}$ and $\ND_p$ do not always agree, notice that for a point $u$  to be displaceable by  a probe it must be 
\lq\lq geometrically visible" from some nearby facet.  On the other hand,
by \cite[\S11]{FOOO1} as $u$ varies in $\De$ the properties of the Floer homology of $L_u$ are governed by the sizes  $\om(\be_i) = \ell_i(u)$ of the discs of Maslov index $2$ that are transverse to $\Phi^{-1}(F_i)$ and have
 boundary on $L_u$.  These holomorphic discs always exist, no matter where $u$ is in $\De$.   Moreover, according to \cite[\S9]{FOOO1}
 in order for $HF_*(L_u,\chi)$ to be nonzero for some perturbation $\chi$ one needs there to be more than one $i$
for which $\ell_i(u)$ is a minimum.   Therefore the set $\ND_{HF}$ always lies in the polytope $P_1$ defined in \S\ref{ss:v0}.

To illustrate these differences in the example of Figure \ref{fig:2},  choose small $\eps>0$ and make the triangle $\De$ smooth by introducing the new edges
\begin{gather}\notag
F_1: =\{ x_1+x_2 = 5-\eps\},\qquad F_2: = \{3x_1+2x_2=10-\eps\},\\\notag
F_3: =\{ 2x_1+x_2=6-\eps\},\qquad F_4: = \{x_1=3-\eps\}.
\end{gather}
Then it is not hard to see that the calculation of $\inf \{\ell_i(u);u\in \intt \De\}$ is dominated by  the distances to the four facets $\{x_1=0\}$, $\{x_2=0\}$, $F_3$  and  $F_4$; the other facets are simply too far away.  In fact, the set $P_1$ 
defined in \S\ref{ss:v0} is the vertical line segment between the points $\bigl(\frac 12(3-\eps),\frac 12(3-\eps)\bigr)$ and
$\bigl(\frac 12(3-\eps),\frac 12(3+\eps)\bigr)$.  In
 other words, as far as the calculation of $v_0$ is concerned,  our polytope might as well be a trapezoid.
\end{rmk}

\begin{rmk}[Reflexive polygons]\labell{rmk:refl}\rm  
Both triangles  in Figure \ref{fig:2b} are reflexive. (They are
Examples 3 and 6d on the list in \cite[\S4]{Nill}.)
 In (I), the direction $(-1,1)$ is integrally transverse to all facets, so that probes in this direction or its negative displace
all points except for those on certain lines.  Because $B$ is smooth, all points near $B$ can be displaced.  But there are lines of nondisplaceable points near
the nonsmooth vertices $A$ and $C$.  The point $(1,1)$ has 
affine distance $1$ 
from each facet, and so is the central point $v_0$. The line segment $v_0D$ 
is contained in the line $x=y$.  By way of contrast, the triangle (II) does not have one direction that is integrally transverse to  all edges, although each pair of edges has an integrally transverse direction. 
 \end{rmk}


\subsection{The $3$-dimensional case}

This section is devoted to proving the following result.

\begin{prop}\labell{prop:3}  Every $3$-dimensional monotone polytope $\De$ satisfies the star Ewald condition. Hence all its points  except for $\{0\}$ are displaceable by probes.
\end{prop}

This result is included  in the computer check by Paffenholz that 
shows that all monotone polytopes of dimension $\le 5$ are star Ewald.
  However, we shall give a more conceptual proof to
   illustrate the kind of ideas that can be used to analyze this problem.

We begin with some lemmas. Throughout, we choose coordinates so that 
 $v = (-1,-1,-1)$ is a vertex of $\De$ and so that
 the facets through $v$ are $F_i: = \{x_i=-1\}$.  We shall say that a facet $F$ of $\De$ is {\it small} if it is a triangle with one (and hence all) edges of length $1$.  

  \begin{figure}[htbp] 
     \centering
    \includegraphics[width=3in]{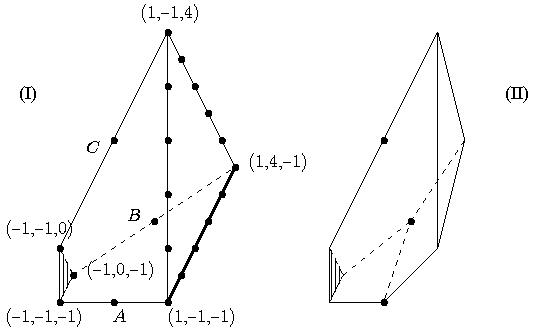} 
     \caption{The truncated pyramid (I) on the left  is a $\De_1$-bundle over $\De_2$ with one small triangular facet; cf. \S\ref{s:bun}.  The integral points on its edges are marked.  Polytope (II) is its monotone blow up along the heavy edge.}
     \label{fig:7}
  \end{figure}

 \begin{lemma} \labell{le:31}If
 $\De$ has a small facet $F$, it is one of the two polytopes illustrated in 
 Fig. \ref{fig:7}.   
 \end{lemma}
 \begin{proof}  We may suppose that the 
vertices  of $F=F_1$ are $v_1=(-1,-1,-1), v_2: = (-1,0,-1)$ and $v_3: = (-1,-1,0)$.  The vertex-Fano condition at $v_2$ implies that the edge through $v_2$ transverse to $F$ must have direction $e_2'=(1,2,0) $ while that
through $v_3$ transverse to $F_1$ must have direction $e_3'=(1,0,2) $.  Therefore $\De\cap \{x_1=0\}$ contains the points  $A:=(0,-1,-1), B: = (0,2,-1)$, and
$ C: = (0,-1,2) $.
\MS
 
  \NI {\bf Claim 1:}   {\it
  If none of $A,B,C$ are vertices, the points $(1,-1,-1), (1,4,-1)$ and $(1,-1,4)$ are vertices of $\De$ and $\De$ is the polytope in part {\rm (I)} of Figure \ref{fig:7}.}
  
  Let $y=(y_1,y_2,y_3)$ be a vertex  on the edges through $A,B,C$ (but not on $F$) with the smallest  coordinate $y_1$.  Without loss of generality we may suppose that $y$ lies on  the edge through  $A$.  By hypothesis, $y_1\ge 1$.  Let $e_{y1}, e_{y2}$  be the primitive vectors along the edges from $y$ that do not go through $A$.  Then the $x_1$-coordinates of the vectors $e_{yi}$ are nonnegative and (by the vertex-Fano condition)  sum to $ -y_1+1\le 0$.
  It follows that  $y_1=1$ and that the vectors $e_{yi}$ lie in the plane $x_1=1$.   It is now easy to see that the polytope must be as illustrated in (I). \MS

\NI {\bf Claim 2:}   {\it
  If at least one  of $A,B,C$ is a  vertex, then $\De$ is the polytope in part {\rm (II)} of Figure \ref{fig:7}.}

Without loss of generality, suppose that
 $A$ is a vertex.  Then, the vertex-Fano condition implies that just one of the
  primitive edge vectors from $A$ has positive $x_1$-coordinate.  In Figure \ref{fig:7} (II) we illustrate the case when this edge vector $e$ lies in $F_2$. Then the  edge through $A$ in $F_3$ ends at $B$, and $e= (1,0,1)$.  It follows that $C$ cannot be a vertex.  Arguing similarly at $B$, we see that
  the slice  $\De\cap \{x_1\in [-1,1]\}$ must equal the polytope (II).  Moreover, 
  if any of the three edges of the slice $\De\cap \{x_1\in (0,1)\}$ have a vertex on $\{x_1=1\}$, then the vertex-Fano condition at this vertex implies that $\De=$ (II).   On the other hand, there must be some vertex of this kind.  
  For if not, we can get a contradiction as in Claim 1 by looking at the vertex  $y$ on these edges with minimal  $x_1$-coordinate.
 \end{proof}
 
 \begin{rmk}\rm  Polytope (II) is the monotone blow up of  (I) along an edge.  Note that to make a monotone blow up along 
an edge $\vareps$ of a monotone $3$-dimensional polytope one must cut out a neighborhood of size $1$.  Thus, there is such a blow up  of $\vareps$
exactly if all edges meeting $\vareps$ have length at least $2$.
Similarly, there is a monotone blow up of a vertex $v$ of a monotone $3$-dimensional polytope 
if all edges meeting $v$ have length at least $3$.  Thus, up to permutation, there is only one monotone blow up of (I).
\end{rmk}

\begin{cor} \label{cor:31}
 A $3$-dimensional polytope can have at most one small facet.
Moreover, if $F$ is small,   $F\cap\De_{\Z} \subset \Ss(\De)$.
\end{cor}
 \begin{proof}  The first statement is obvious; the second follows by 
inspection.
 \end{proof}

 Given a vertex $v$ in some facet $F$ of $\De$  we
define the {\it special point  $s_{v,F}$ of $v$ in the facet } $F$ to be $v+\sum e_j$ where $e_j$ ranges over the primitive integral vectors along the edges from $v$ that lie in $F$. With our choice of coordinates, we get for $v=(-1,-1,-1)$ 
 the three points  $s_1: = s_{v,F_1}= (-1,0,0), s_2: = s_{v,F_2}=(0,-1,0)$ and $s_3: = s_{v,F_3}=(0,0,-1)$. Note that:\MS
 
 \NI $\bullet$
  $v$ satisfies the star Ewald condition if and only if one of these three points lies in $\Ss(\De)$;\SSS
  
  \NI $\bullet$  $s_{v,F}\in \De$ unless $F$ is  a small facet.
 \MS
 
   Next, given $v\in F_i$ we define {\it the facet  opposite to $F_i$ at $v$}
  as follows:
if $\vareps_i$ is the edge from $v$ transverse to $F_i$, and $z_i$ is the other vertex of $\vareps_i$ then $F_i'$ is the facet through $z_i$ not containing $\vareps_i$.  Consider the special point $s'_i=s_{z_i,F_i'}$.  Then $s_i'=-s_i$ because the vertex-Fano condition implies that $s_i+e_i = 0 = s_i' - e_i$.  Therefore if there is $i$ such that $s_i, s_i'$ are both in $\De$,
 the star Ewald condition is satisfied at $v$.\MS

      \NI {\bf Proof of Proposition \ref{prop:3}.}
  \SSS
  
   \NI {\bf Step 1:} {\it Every vertex of $\De$ satisfies the star Ewald condition.}
   
Without loss of generality, consider the vertex $v=(-1,-1,-1)$ as above.
If neither of $F_1, F_1'$ are small then $s_1\in \Ss(\De)$ and $v$ is star Ewald.  But if one of these facets is small, then  Lemma~\ref{le:31}  implies that
$\De$ is one of the polytopes in  Figure \ref{fig:7} and one can check directly that in these cases  every vertex lies on some non-small facet $F$ whose opposite facet at $v$ is also not small.\MS

  \NI {\bf Step 2:} {\it Every facet $F$ of $\De$ satisfies the star Ewald condition.}

If $F$ is small, then its three vertices
 lie in $\Ss(\De)$ by Corollary \ref{cor:31}.  Otherwise,
let $F_w'$ be the facet opposite to $F$ at some vertex $w\in F$.  
If $F'_w$ is small for all choices of $w$  then 
 $F$ must be the large triangular facet in  the polytope (I) of Figure \ref{fig:7},
 and so it contains points in $\Ss(\De)$ by Corollary \ref{cor:31}.
 The remaining case is when there is $w$ such that 
 $F_w'$ is not small.  But then  $s_{w,F}\in \Ss(\De)$. 
\MS

    \NI {\bf Step 3:} {\it Every edge $\vareps$ of $\De$ satisfies the star Ewald condition.}

If $\De$ has no small facets, we may suppose
that $\vareps= F_2\cap F_3$.
The proof of Step 1 shows that
$s_{v,F_2}= (0,-1,0)\in F_2$ lies in $\Ss(\De)$.   Further, because 
$\starr(\vareps) = \vareps$,  the conditions  $s_{v,F_2}\in 
\istar(\vareps)$ and $-s_{v,F_2}\in \Starr(\vareps)$ are obviously satisfied. 

 On the other hand,  if $\De$ has a small facet, then $\De$ is (I) or (II).
In these cases, one can check that for  each edge of $\De$ at least one of the two facets containing it, say $F$, is not small and 
has an opposite facet $F' $  at some point  $w\in F$ that is also not small.  Therefore, again $s_{w,F}\in \Ss(\De)$ satisfies the star Ewald condition at $\vareps$.

This completes the proof of Proposition \ref{prop:3}.

To extend this kind of argument to higher dimensions, one would have to 
understand 
facets that are small in the sense that they
 do not contain some (or all) of their special points $s_{v,F}$.

\section{Bundles.}\labell{s:bun}

In \cite{OT} Ostrover--Tyomkin construct an $8$-dimensional monotone toric manifold $M_{OT}$ whose quantum homology is not semisimple.
  As pointed out in \cite{FOOO1}  the properties of $QH_*(M)$ are closely related to the nondisplaceable points in $\De$.  Nevertheless, we show that in the case of $M_{OT}$ the special fiber is a stem. The manifold $M_{OT}$ is a toric bundle over $\C P^1\times \C P^1$ with fiber the $3$-point blow up of $\C P^2$.  Hence this example is covered by
Corollary \ref{cor:bun} below.
 
 \begin{figure}[htbp] 
    \centering
    \includegraphics[width=3.5in]{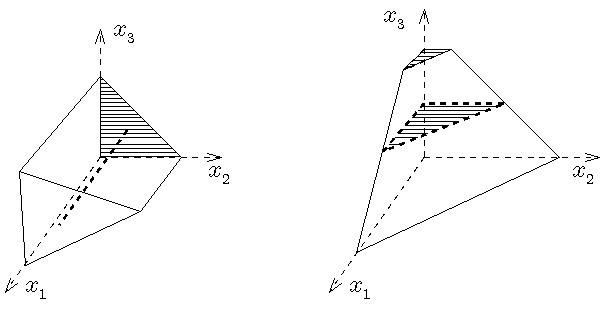} 
    \caption{The polytope (a)
  is a $\De_2$-bundle over $\De_1$, while  (b) is a
  $\De_1$-bundle over $\De_2$.  The shaded facet in (a)  is one of the base facets $\Hat F_i'$ and is affine equivalent to the fiber, while the top shaded
  facet in (b) represents a section of the bundle and is one of 
  the two fiber facets $\Tilde F_j'$.
  The heavy dotted lines enclose the 
  central slice $\Hat\De_0$ described in Lemma~\ref{le:bun}.}
    \labell{fig:6}
 \end{figure}

 Recall that a smooth locally trivial fiber bundle  $\Tilde M\to M\stackrel{\pi}\to \Hat M$ whose total space is a toric manifold $(M,T)$ is said to be a toric bundle if 
the action of $T$ permutes the fibers of $\pi$.  It follows that
 there is a corresponding quotient homomorphism $T\to \Hat T: = T/\Tilde T$ whose kernel $\Tilde T$ induces a toric action on the fiber $\Tilde M$ and whose image $\Hat T$ induces a toric action on the base $\Hat M$.  Because the moment polytope $\De$ lives in the dual Lie algebra $\ft^*$  of $T$,  there is no natural projection map from $\De\subset \ft^*$ to $\Hat \De\subset {\Hat \ft}\,\!^*$.  Rather, 
 the projection $\pi$ induces an inclusion
$\pi^*: {\Hat \ft}\,\!^*\to \ft^*$, 
and the natural projection is $ \ft^*\to \Tilde\ft^*$. 
From now on, we will  
identify $\ft^*$ with ${\Hat\ft}\,\!^*\times {\Tilde\ft}\,\!^* = \R^k\times \R^m$ in such a way that $\R^k=\pi^*({\Hat\ft}\,\!^*)$, and 
will translate $\De$ so that its special point is at $\{0\}$.
 This implies  that $\De$ is obtained by slicing a polytope $\Tilde\De'$ (which is affine isomorphic to 
 $\R^k\times \Tilde\De$ and so has facets parallel  to the first $k$-coordinate directions) by hyperplanes that are in bijective correspondence with the facets of the base $\Hat\De$.  

 Here is the formal definition of bundle.  Note that here we consider the normals to the facets as belonging to the Lie algebra of the appropriate torus.

 \begin{defn}\labell{def:bund} 
Let  $\Hat{\Delta} =
\bigcap_{i = 1}^{\Hat N} \{ x \in \Hat{\ft}^* \mid \langle \Hat{\eta}_i,x
\rangle \leq \Hat{\kappa}_i \}$ and
$\Tilde{\Delta} =
\bigcap_{j = 1}^{\Tilde N} \{ y \in \Tilde{\ft}^* \mid \langle \Tilde{\eta}_j,y
\rangle \leq \Tilde{\kappa}_j \}$ 
be simple polytopes.
We say that a simple 
polytope 
$\Delta \subset \ft^*$
is a {\bf bundle} with {\bf fiber} $\Tilde{\Delta}$ over the {\bf
base} $\Hat{\Delta}$ if 
there exists  a short exact sequence
$$ 
0 \to \Tilde{\ft} \stackrel{\iota}{\to} \ft \stackrel{\pi}{\to} \Hat{\ft} \to 0
$$
so that
\begin{itemize}
\item 
$\Delta$ is combinatorially equivalent to the product $\Hat{\Delta}
\times \Tilde{\Delta}$.
\item 
If $\Tilde{\eta}_j\,\!'$   
 denotes the outward normal to
the  facet 
$\Tilde{F}_j\,\!'$
of $\Delta$ which corresponds to  $\Hat{\Delta} \times
\Tilde{F_j} \subset \Hat{\Delta} \times \Tilde{\Delta}$,
then $\Tilde{\eta}_j\,\!' = \iota(\Tilde{\eta}_j)$ 
for all $1 \leq j\leq \Tilde N.$
\item 
If $\Hat{\eta}_i\,\!'$  
denotes the outward normal to
the  facet 
$\Hat F_i\,\!'$
of $\Delta$ which corresponds to  $\Hat{F}_i \times
\Tilde{\Delta}
 \subset \Hat{\Delta} \times \Tilde{\Delta}$,
then $\pi(\Hat{\eta}_i\,\!') = \Hat{\eta}_i$
for all $1 \leq i \leq \Hat N.$
\end{itemize}
 The facets 
$\Tilde{F}_1\,\!', \ldots, \Tilde{F}_{\Tilde N}\,\!'$ correspond bijectively to the facets of the fiber and 
will be called {\bf fiber facets},  
while the facets 
$\Hat F_1\,\!' \ldots, \Hat F_{\Hat N} \,\!'$ (which 
correspond bijectively to the facets of the base)
will be called {\bf base facets}.
\end{defn}

The fiber facets $\Tilde F_j\,\!'$ are all parallel to the $k$-plane $\pi^*({\Hat\ft}\,\!^*)$ in $\ft^*$ and (when $m>1$) are themselves bundles over
$\Hat\De$.  For each vertex $\Tilde v_\be = \cap_{j\in \be}\Tilde F_j$ of $\Tilde \De$, the intersection 
$$
\bigcap_{j\in \be}\Tilde F_j\,\!'=:\Hat\De_{\Tilde v_\be}
$$ 
of the corresponding fiber facets is an $k$-dimensional polytope,
 that is parallel to $\pi^*({\Hat\ft}\,\!^*)$ and projects to a polytope 
in ${\Hat\ft}\,\!^*$  that is {\it analogous} to $\Hat\De$. In other words
$\Hat\De_{\Tilde v_\be}$ has the same 
normals $\Hat\eta_i$ as $\Hat\De$ but usually {\it different} structure constants $\Hat\ka_i$.
For example, in the polytope (a) in
 Figure \ref{fig:6} the $\Hat\De_{\Tilde v_\be}$ are 
 edges of various lengths that are parallel to the $x_1$-axis, while on the right
they are the top and bottom  triangles. In contrast, it is not hard to see that the faces
$$
\bigcap_{i\in \al}\Hat F_i\,\!'=: \Tilde \De_{\Hat v_\al}
$$
 of $\De$ corresponding to the 
vertices
$\Hat v_\al = \cap_{i\in \al}\Hat F_i$ of the base are all affine equivalent
to the fiber $\Tilde\De$.
 Thus, if a given polytope $\De$ is a $\Tilde\De$-bundle over $\Hat\De$, 
 the fiber  polytope $\Tilde\De$ is completely determined by $\De$ while the base polytope is only determined modulo the structure constants (though these must remain in the same chamber so that the intersection pattern of the facets, i.e. the fan, does not change).

Observe also that the polytope $\De$ is the union of
$k$-dimensional  parallel slices 
\begin{equation}\labell{eq:slice}
\Hat \De_y: =\De\cap (\R^k\times \{y\}),\quad y\in \Tilde\De;
\end{equation}
cf. Figures \ref{fig:6} and \ref{fig:10}.
The following lemma shows that if $\De$ is monotone then so
 is its fiber.  Moreover, as we see in parts (ii) and (iii), 
$\De$  also determines 
 a particular monotone base polytope.
 
\begin{lemma}\labell{le:bun}  Suppose that the monotone polytope $\De$ is a $\Tilde\De$ bundle over $\Hat\De$ with special point $u_0 = 0$.  Then:\SSS

\NI {\rm (i)} The fiber $\Tilde\De$ is monotone.\SSS

\NI {\rm (ii)} For each $y\in \Tilde\De$, the slice $\Hat\De_y$ in equation (\ref{eq:slice}) does not depend on the choice of splitting $\ft^*=
{\Hat\ft}\,\!^*\times {\Tilde\ft}\,\!^*.$  It is integral whenever $y$ is.

\SSS

\NI {\rm (iii)}
The slice 
$$
\Hat\De_0 =\De\cap\bigl(\R^k\times \{0\}\bigr)
$$ 
through the special point of $\Tilde\De$ is monotone.\SSS

\NI {\rm (iv)} For any $w\in \Ss(\Tilde\De)$
the intersection $\Hat\De_{w\R}: = \De\cap\bigl(\R^k\times \{cw:c\in [-1,1]\}\bigr)$ is a (smooth, integral) $\De_1$-bundle over $\Hat\De$ and is monotone.
\end{lemma} 
\begin{proof}   Consider a
vertex $v_{\al\be}$ of $\De$ corresponding to the pair of vertices
 $\Hat v_\al,\Tilde v_\be$,
where $\Hat v_\al = \cap_{i\in \al} \Hat F_i\in \Hat \De$ for some $k$ element subset $\al\subset\{1,\dots,\Hat N\}$ and similarly for  $\Tilde v_\be\in \Tilde\De$.
The edge vectors at $v_{\al\be}$ divide into two groups.  There are $k$ primitive edge  vectors $
 e_{i}^{\al\be}$ in the plane $\R^k\times\Tilde v_\be$ that are parallel to the edge vectors of $\Hat\De$ through $\Hat v_\al$ (and hence are independent of the choice of $\Tilde v_\be$).  These lie in the face
 $\Hat \De_{\Tilde v_\be}$.  Similarly, there are $m$ others  $e_{j}^{\al\be}$
that lie in the face
$\Tilde \De_{\Hat v_\al}$ and project to the edge vectors $\Tilde e_{j}\,\!^{\be}$ of $\Tilde\De$ through $\Tilde v_\be$, but may also have nonzero components $\Hat e_j\,\!^{\al\be}$ in the $\R^k$ direction that depend on $\al,\be$. 
Moreover, we may label these edges 
  so that $e_{i}^{\al\be}$ is transverse to the facet $\Hat F_i\,\!'$ for each $i\in \al$
  and
$e_{j}^{\al\be}$ is transverse to the facet $\Tilde F_j\,\!'$ for each
$j\in \be$.  

Now consider (i).
Since
$
 v_{\al\be}+\sum _{i\in\al} e_{i}^{\al\be} + \sum_{j\in\be} e_{j}^{\al\be} =
  0\in \R^n
 $, we find
$$ 
w_{\al\be}: =v_{\al\be}+\sum _{j\in\be} e_{j}^{\al\be} = (\Hat w_{\al\be},0),\quad \mbox{
 where}\;\;\;
 \Hat w_{\al\be} +\sum_{i\in\al} e\,\!_{i}^{\al\be}=0.
 $$
In particular, the projection of $v_{\al\be}+\sum _{j\in\be} e_{j}^{\al\be}$ onto $\Tilde \ft^*$  vanishes.  Hence each vertex of $\Tilde \De$ satisfies the vertex-Fano condition with respect to $\{0\}$.  Therefore
 $\{0\}\in \intt\Tilde\De$, and by Remark \ref{rmk:inter}
  it must be  the unique   interior integral point in  $\Tilde \De$.
This proves  (i). (Note that (i) is immediately 
 clear if one thinks of the corresponding fibration of toric manifolds.)

The first statement in (ii) holds because the $\Hat\De_y$ are the intersections of $\De$ with the $k$-dimensional affine planes parallel to $\pi^*({\Hat\ft}\,\!^*)$ and so do not  depend on any choices.  

Now suppose that $y$ is integral. 
There is one vertex of $\Hat \De_y$ corresponding to each vertex $\Hat v_\al$ of  $\Hat\De$, namely  the intersection
$\Hat \De_y\cap \Tilde \De_{\Hat v_\al}$, and we must show that this is integral.
 Since $\Tilde\De$ is monotone, 
$y$ can be written as $\Tilde v_\be+\sum k_j\Tilde e_j$, where  $\Tilde v_\be$ is a vertex, $k_j\ge 0$ are integers,
and the  $\Tilde e_j$ are the primitive  integral  edge vectors at $\Tilde v_\be$. 
As explained above, there is a vertex $v_{\al\be}$ of $\De$ corresponding to the pair $\al,\be$, and there are corresponding
 primitive edge vectors  $e_j^{\al\be}$ of $\De$ at  $v_{\al\be}$ that project onto 
$\Tilde e_j$ and lie in the face  $\Tilde \De_{\Hat v_\al}$.  It follows immediately that  $v_{\al\be}+\sum k_je_j^{\al\be}$ is an integral point that  lies in
$\Tilde \De_{\Hat v_\al}$ and projects to $y$.  Hence it equals the intersection $\Hat \De_y\cap \Tilde \De_{\Hat v_\al}$.  This proves (ii).

Claim (iii) will follow once we show that $w_{\al\be}=(\Hat w_{\al\be},0)$ is the vertex of $\Hat\De_0$ corresponding to $\Hat v_\al$.  
But, by construction, each edge $e_{j}^{\al\be}$ lies in every facet 
$\Hat F_i', i\in \al$, as does $v_{\al\be}$. Therefore 
%
$$
 w_{\al\be}: =v_{\al\be}+\sum _{j\in\be} \Hat e_{j}\,\!^{\al\be}\in 
\bigcap_{i\in \al} \Hat F_i\,\!'.
 $$
 Thus it projects to $\Hat v_\al$ as claimed. This proves (iii).
 
 \begin{figure}[htbp] 
    \centering
\includegraphics[width=4in]{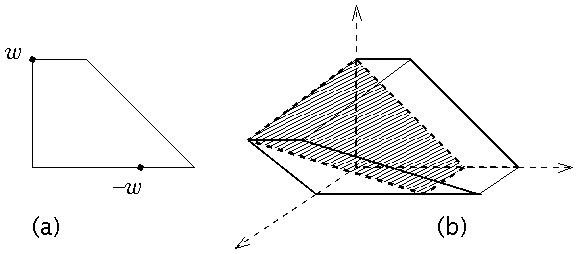} 
    \caption{ (a) shows the two points
     $\pm w\in \Ss(\Tilde\De)$
    where $\Tilde\De$ is the monotone one point blow up of $\De_2$.
     (b) illustrates  a $\Tilde \De$ bundle over $\De_1$, with the corresponding $\De_1$-bundle $\De_{w\R}$ shaded.  }
    \label{fig:10}
 \end{figure}
 Now consider (iv).  It is geometrically clear that $\Hat\De_{w\R}$ is a 
 $\De_1$-bundle over $\Hat\De$; cf. Figure \ref{fig:10}. It is integral by (ii).    
 
 To see that it is monotone, we shall check the vertex-Fano condition. To this end, note that for every vertex $\Hat v_\al$ of $\Hat\De$ there are two vertices of 
$\Hat\De_{w\R}$. Call them $v_{\pm\al}$, where  $v_{+\al}\in \Hat\De_w$.  Denote by $e$ the primitive vector along the edge from $v_{+\al}$ that 
 projects to $-w\in \R^m$.  
Then $ v_{+\al} + e=:\Hat v_{0\al}\in \Hat\De_0$.  The other edge vectors $\Hat e_j\,\!^{+\al}$ of 
$\Hat\De_{w\R}$ through $v_{+\al}$ are parallel to the edge vectors 
${\Hat e}_j\,\!^{\al0}$ of $
\Hat\De_0$ at the vertex $\Hat v_{\al0}$.
Hence  $$
\Hat v_{\al0} + \sum_j {\Hat e}_j\,\!^{\al 0} = \Hat v_{+\al} + e + \sum_{j\in \al }\Hat e_{j}\,\!^{+\al}=\{0\}
$$
 as required.  A similar argument applies to $v_{-\al}$.
\end{proof}

It follows that  if a monotone polytope $\De$ is a bundle we may identify its base with the slice $\Hat\De_0$.  Since its fiber and its base are monotone, it makes sense to consider the star Ewald condition for these polytopes.

 \begin{prop}\labell{prop:bun1}  Let $\De$ be a monotone polytope
that is a $\Tilde\De$-bundle over $\Hat\De$. Suppose that $\Tilde\De$
and all monotone $\De_1$-bundles over $\Hat\De$ satisfy the
  star Ewald condition.  Then 
 $\De$ satisfies the 
 star Ewald condition.
 \end{prop}

\begin{proof}   Lemma~\ref{le:bun} for $\De$ implies that $\Hat\De$ is monotone.  Hence so is the product $\Hat\De\times \De_1$ where $\De_1=[-1,1]$ is monotone.  Using the fact that $\Hat\De\times \De_1$ is star Ewald, one easily deduces that $\Hat\De$ is as well.

Now, consider the face 
$$
f= \left(\bigcap_{i\in \al} \Hat F_i'\right)\cap
\left(\bigcap_{j\in \be} \Tilde F_j'\right)=: F_{\al\be}
$$
of $\De$.
Then $\Starr f \cap \Hat\De_0\ne \emptyset$ provided that
$\al\ne \emptyset$.   Further, 
because $\{0\}$ is in the interior of all faces of $\Tilde\De$, the intersection
$\starr f \cap \Hat\De_0$ may be identified with  $\starr \Hat F_\al$ in $\Hat \De$,
where, as usual, $\Hat F_\al: = \cap_{i\in \al}\Hat F_i$.  Thus,
 when $\al\ne \emptyset$ the star Ewald condition for 
the face $\Hat F_\al$ in $\Hat\De$  implies that for $F_{\al\be}$ in $\De$.

Now consider $F_\be: = F_{\emptyset\be},$ and write $\Tilde F_\be = \cap_{j\in \be} \Tilde F_j$.  Thus $F_\be\subset \R^k\times\Tilde F_\be$ is a union of the slices $\Hat\De_y, y\in \Tilde F_\be$.
By the star Ewald condition for $\Tilde\De$ there is $w\in   \Ss(\Tilde\De)$ such that $w \in \istar \Tilde F_\be$ and $-w\notin \Starr \Tilde F_\be$.
Then $\istar F_\be$ contains the facet $\Hat\De_w$ of the polytope $\Hat\De_{w\R}$ considered in part (iv) of Lemma~\ref{le:bun}.
(In fact, $\istar F_\be$ can be identified with the union of all slices $\Hat\De_y$ with $y\in \istar \Tilde F_\be$.)  By hypothesis,
the facet $\Hat\De_w$
 has the star Ewald condition in $\Hat\De_{w\R}$. That is,  there is an element $\la\in\Ss(\Hat\De_{w\R})$ in $\Hat\De_w$. By construction,
 $\la$ has the form $(x,w)$ for some $x\in \R^k$.   Therefore $-\la=(-x,-w)$ projects to $-w\not\in \Starr \Tilde F_\be$ and so $-\la\not\in
 \Starr  F_\be$.  Since  
 $\Ss(\Hat\De_{w\R})\subset \Ss(\De)$, the result follows.
 \end{proof}
 
 The following result was proved in the course of the above argument.
 
 \begin{lemma}
 If $\Hat\De$ is star Ewald, then a $\De_1$-bundle over $\Hat\De$ is also star Ewald provided that 
one (and hence both) of its fiber facets contains a symmetric point, i.e. intersects $\Ss(\De)$.  
\end{lemma}

   \begin{cor}  A product $\Hat\De\times \Tilde\De$ of monotone polytopes is star Ewald if and only if its two factors are.
   \end{cor}
   
   \begin{proof}  This is easy to check directly.  However, the proof of
   Proposition \ref{prop:bun1} allows one to reduce the
 proof of the \lq\lq if" statement to the case of $\De_1\times \De_1$.
 \end{proof}
 
 \begin{cor} \labell{cor:bun} Let $\Tilde\De$ be star Ewald.  Then every monotone $\Tilde\De$-bundle over the simplex $\De_k$ is star Ewald.
 \end{cor} 
 \begin{proof}  By Proposition \ref{prop:bun1} it suffices to consider the case $\Tilde\De=\De_1$, and to show in this case that the two fiber facets
 $\Hat\De_{\pm}$ are star Ewald. We may choose coordinates  $(x,y)\in \R^k\times \R$ so that $\De$ is given by
 the inequalities
 $$
-1\le y\le 1,\;\;  x_i\ge -1, i=1,\dots,k, \;\;\sum_{i=1}^{k}x_i\le 1+\al y,
$$
for some integer $\al\ge 0$.
 Then
 the fiber facets $\Hat\De_+ = \{y=1\}$ and $\Hat\De_- = \{y=-1\}$ are  integral 
 simplices with side lengths $k+1+\al$ and $k+1-\al
 $ respectively.  Since the case $\al=0$ is obvious we may suppose that $1\le \al\le k$.   
Let  $v=(v_1,\dots,v_{k+1})$ be a point with
$v_{k+1}=1$ and with precisely $\al$ of the coordinates $v_1,\dots, v_{k}$ equal to $1$, while the others are $0$. 
Then 
$v\in \Hat\De_+$ while $-v\in \Hat\De_-$.   
\end{proof}
  
  \begin{rmk}\rm  (i)  One should be able to prove the analog of Corollary \ref{cor:bun} for all bases $\Hat\De$.  However, this does not seem easy.
The following discussion  explains what the problem is, and proves a special case.

 Let $\Hat\eta_i, i\in \{1,\dots,\Hat N\},$ be the (outward) normals of a smooth $k$-dimensional 
  polytope $\Hat\De$, chosen so that  the first $k$ are the negatives of a standard basis.
Then, with $(x,y)$  as above, every $\De_1$-bundle over $\Hat \De$ 
has the form
$$
-1\le y\le 1,\;\;  x_j\ge -1, j=1,\dots,k, \;\;\langle x,\eta_i\rangle \le \ka_i-a_iy, \;i>k.
$$
If $\Hat\De$ is monotone, we may take $\ka_i=1$ for all $i$ so that $\De$ is determined by the $\Hat N-k$ integers  $a_i, i>k$, where we set $
a_i=0,i\le k$.   It is easy to check that $\De$ is monotone provided that it is combinatorially equivalent to a product.  This will be the case exactly when the top and bottom facets 
$\{y=\pm1\}$ of $\De$ are analogous. However, it is not clear 
in general what conditions this imposes on the constants $a_i$, except that they cannot be too large.\footnote
{
The  constants $a_i$ determine and are determined by
  the Chern class of the corresponding $\PP^1$ bundle over $\Hat M$.}
  More precisely,  if the edge $e_{i\ell}$ meets the facets $\Hat F_i$ and $\Hat F_\ell$ transversally, 
  then then, because the length $L(e_{i\ell}^{\pm})$ of the corresponding edges in $\{x=\pm1\}$ must be at least $1$, we must have
  $$
  L(e_{i\ell})  \ge |a_i-a_\ell| + 1,\quad \forall i,\ell.
  $$
Thus, if $\Hat\De$ has many short edges then it supports  few monotone bundles.
    One easy case is when the $|a_i|$ are all $\le 1$.   For then the points $(0,\dots,0,\pm 1)$ lie in $\Ss(\De)$ so that $\De$ is star Ewald exactly if $\Hat\De$ is.

One can rephrase  these conditions by making a different choice of coordinates.
If there is a symmetric point $w$ in $\{y=1\}$ then we can use $w$ instead of
$(0,\dots,0,1)$ as the last basis vector, keeping the others unchanged.  
Then $\De$ is given by  equations of the form
\begin{equation}\labell{eq:Hat}
-1\le y\le 1,\;\;  x_i\ge -1+b_i y, i=1,\dots,k, \;\;\langle x,\eta_i\rangle \le 1-b_iy, \;i>k,
\end{equation}
where we must have $|b_i|\le 1$ for all $1\le i\le \Hat N$ because now $(0,\dots,0,\pm 1)\in \De$ by our choice of coordinates.  It follows easily that, assuming $\Hat \De$ is star Ewald, then $\De$ is star Ewald if and only if it may be given by equations of the form (\ref{eq:Hat}).
\MS

\NI
(ii)  The problem considered in (i) above is a special case of the following question.
Consider a  symplectic $S^2$ bundle $S^2\to (X,\om)\stackrel{\pi}\to \Hat X$  with symplectic base  $(\Hat X,\Hat \om)$, where we assume that $\om$ is nondegenerate on each fiber.   If a subset $L\subset \Hat X$ is displaceable by a Hamiltonian isotopy, is it true that  its 
inverse image $\pi^{-1}(L)$ is displaceable in $(X,\om)$?  
At first glance, one might think this is obviously true.  However one cannot assume that
there is a simple relation between $\om$ and $\pi^*(\Hat\om)$, and so there is no obvious way to lift an isotopy of $\Hat X$ to one of $X$.  In fact the awkwardness of the definition of toric bundle is one indication of the subtlety of this relation.  \end{rmk}  

\end{document}